\documentclass[reqno]{amsart}
\usepackage{amssymb}
\usepackage[mathscr]{eucal}
\usepackage{hyperref}
\usepackage{epstopdf}

\usepackage[all]{xy}
\usepackage{color}
\usepackage[margin=4cm]{geometry}

\usepackage{tikz}
\usetikzlibrary{cd} 
\usetikzlibrary{babel} 

\theoremstyle{plain}
\newtheorem{prop}{Proposition}[section]
\newtheorem{thm}[prop]{Theorem}
\newtheorem{cor}[prop]{Corollary}
\newtheorem{lem}[prop]{Lemma}

\theoremstyle{definition}
\newtheorem{dfn}[prop]{Definition}
\newtheorem{rem}[prop]{Remark}
\newtheorem{example}[prop]{Example}

\renewcommand{\iff}{\Leftrightarrow}

\newcommand{\A}{{\mathbb{A}}}

\newcommand{\R}{{\mathbb{R}}}
\newcommand{\Z}{{\mathbb{Z}}}
\newcommand{\T}{\mathbb{T}}
\newcommand{\cM}{\mathcal{M}}

\newcommand{\abs}{\vert \,.\, \vert}

\newcommand{\scrF}{{\mathscr{F}}}

\newcommand{\Hscr}{\mathscr{H}}

\DeclareTextFontCommand{\textnf}{\normalfont}

\DeclareMathOperator{\Hom}{Hom}
\DeclareMathOperator{\im}{Im}

\DeclareMathOperator{\Spec}{Spec}
\DeclareMathOperator{\Sper}{Sper}

\DeclareMathOperator{\supp}{supp}
\DeclareMathOperator{\Trop}{Trop}

\DeclareMathOperator{\an}{an}
\DeclareMathOperator{\trop}{trop}

\DeclareMathOperator{\val}{\nu}
\DeclareMathOperator{\Int}{int}

\newcommand{\sgn}{{\rm sgn}}
\newcommand{\Xan}{X^{\an}}
\newcommand{\Xanpm}{X_r^{\an}}

\newcommand{\Sanpm}{S_r^{\an}}
\newcommand{\Khat}{\hat K}

\renewcommand{\emptyset}{\varnothing}
\renewcommand{\setminus}{\smallsetminus}
\newcommand{\ol}{\overline}

\newcommand{\To}{\Rightarrow}

\renewcommand{\subset}{\subseteq}
\renewcommand{\supset}{\supseteq}

\newcommand{\interior}[1]{\Int\left( #1\right)}

\newcommand{\Label}[1]{\label{#1}}

\newcommand{\marginnote}[1]{\vrule width0pt height0pt depth0pt
    \vadjust{\vbox to0pt{\vss\hbox to\hsize{\hskip\hsize\quad
    #1\hss}\vskip1.5pt}}}

\newcommand{\arch}{\mathrm{arch}}
\newcommand{\mx}{\mathrm{max}}

\begin{document}

\title{Real tropicalization and analytification of semialgebraic sets}

\author{Philipp Jell}
\address{Universit\"at Regensburg, 93040 Regensburg, Germany}
\email{philipp.jell@ur.de}

\author{Claus Scheiderer}
\address{Fachbereich Mathematik und Statistik, Universit\"at Konstanz,
78457 Konstanz, Germany}
\email{claus.scheiderer@uni-konstanz.de}

\author{ Josephine Yu}
\address{School of Mathematics, Georgia Tech, Atlanta GA 30332, USA}
\email{jyu@math.gatech.edu}
\thanks{PJ was supported by the DFG Research Fellowship  JE 856/1-1,
CS was supported in part by DFG grant SCHE281/10-2, and
JY was supported in part by NSF grant DMS-1600569.}

\subjclass[2010]
{Primary 14P10,
secondary 14T05, 14G22, 32P05}


\begin{abstract}
Let $K$ be a real closed field with a nontrivial non-archimedean
absolute value.
We study a refined version of the tropicalization map, which we call
real tropicalization map, that takes into account the signs on $K$.
We study images of semialgebraic subsets of $K^n$ under this map
from a general point of view.
For a semialgebraic set $S \subset K^n$ we define a space $\Sanpm$
called the real analytification,
which we show to be homeomorphic to the inverse limit of all real
tropicalizations of $S$.  We prove a real analogue of the tropical
fundamental theorem and show that the tropicalization of any semialgebraic
set is described by tropicalization of finitely many inequalities
which are valid on the semialgebraic set. We also study the topological
properties of real analytification and tropicalization.
If $X$ is an algebraic variety, we show that  $\Xanpm$ can be
canonically embedded into the real spectrum $X_r$ of $X$, and we study its relation with the Berkovich analytification of $X$.
\end{abstract}

\maketitle


\section{Introduction}

Let $K$ be an algebraically closed field which is equipped with a non-archimedean non-trivial absolute value $\abs$.
The map
\begin{align}
\trop \colon K^n \to (\R \cup \{-\infty\})^n,\quad
(x_1,\dots,x_n) \mapsto (\log \vert x_1 \vert, \dots, \log \vert x_n \vert)
\end{align}
is called the tropicalization map.
For a closed algebraic subvariety $X$ of $K^n$, the associated tropical variety $\Trop(X)$
is defined as the closure in $(\R \cup \{-\infty\})^n$ of $\trop(X)$.
In good situations, this tropical variety contains information about $X$, such as
its dimension, degree, Chow cohomology class~\cite[\S
6.7]{TropicalBook} or the valuation of the $j$-invariant when $X$ is
an elliptic curve~\cite{KMM}.

Given a variety $X$ over $K$, Berkovich defines the analytification $\Xan$ of $X$,
which is a connected Hausdorff topological space  that contains
$X(K)$ as a dense subset~\cite{Berkovich}.
If $X$ is a closed subvariety of $\A^n$, the tropicalization map above naturally extends to $\Xan$.
The fundamental theorem of tropical geometry states that $\Trop(X) = \trop(\Xan)$ and that $\Trop(X)$ can  be characterized by tropicalizing
the polynomials that vanish on $X$ (see \cite[Theorem 4.2]{Draisma}, \cite{EKL}, \cite[\S 3.2]{TropicalBook}).
Given an affine variety over $K$, Payne shows that  $\Xan$ is homeomorphic to the inverse limit of tropicalizations
of all embeddings of $X$ into affine spaces~\cite[Theorem 1.1]{p}.

In this paper, we consider a similar setup, but with the following
modifications and generalizations:
\begin{enumerate}
\item
$K$ is real closed instead of algebraically closed;
\item
the tropicalization takes the order on $K$ into account;
\item
$X$ does not need to be an algebraic variety but may be only a semialgebraic set.
\end{enumerate}
Our goal is to define a space $\Xanpm$, which we call \emph{real
  analytification} of $X$,
and prove analogues of the fundamental theorem and the limit theorem for it.

For an algebraic variety we define $\Xanpm$ in Definition~\ref{Xanpm} and for a semialgebraic set $S \subset K^n$,
we define $\Sanpm$ in Definition~\ref{dfnmanpm}.
We prove an analogue of the tropical fundamental theorem in Theorem~\ref{thm:fund} and
the analogue of Payne's theorem in Theorem~\ref{signed limit}.

Let us provide some more details.
Let $K$ be a real closed field with a nontrivial
non-archimedean absolute value $\abs_K$ that is compatible with the order on $K$.
An example is the field of real Puiseux series
$\R \{\!\{t \}\!\} = \bigcup_{n \in \mathbb{N}_{>0}} \R((t^{1/n}))$,
where the positive Puiseux series are those with positive leading coefficients.
The real tropicalization map is defined as
\begin{align*}
\trop_r \colon K^n \to \R^n,\quad
(x_1,\dots,x_n) \mapsto (\sgn(x_1) \vert x_1 \vert_K,\dots, \sgn(x_n) \vert x_n \vert_K).
\end{align*}
For a semialgebraic set $S \subset K^n$, the real tropicalization $\Trop_r(S)$ is defined to be
the image of $S$ under $\trop_r$ after extending scalars to a field
with value group $\R$.
The logarithmic version of this construction, without signs, was first
used by Alessandrini in~\cite{Ale}.  For polytopes this was done
earlier in~\cite{DevelinYu}.  We do not know of general
algorithms for computing tropicalizations of real varieties or
semialgebraic sets, but there are some descriptions of tropicalized
spectrahedra~\cite{AGS} and numerical methods in the case of
curves defined over $\R$~\cite{BHV}. 

Since we do not want to restrict to a fixed orthant, we choose to omit the logarithm to
simplify notation and to make the topology more apparent.
This means that our $\Trop_r(S)$ is not piecewise linear,
but it is more meaningful topologically since we are taking into account all
orthants, similarly to Viro's patchworking for hypersurfaces~\cite{Viro_curves}.
For example, Celaya showed recently that the real tropicalization
of a tropical linear space defined over $\R$, and more generally the
real Bergman fan of an oriented matroid,
is the cone over a piecewise linear topological sphere in
$\R^n$~\cite{Celaya}. We will see in Corollary~\ref{cor:tropConnected} that if $S$ is semialgebraically
connected, then its real tropicalization is connected.

For an affine $K$-variety $X=\Spec(A)$, the Berkovich analytification
$X^{\an}$ of $X$ consists of the $K$-seminorms on $A$, and the real
spectrum $X_r$ consists of the orderings on~$A$. We will review these
notions in more detail in~\S\ref{sec:background}. In \S\ref{sec:Xanpm}
we define the real analytification $\Xanpm$ as a topological space consisting of {\em signed seminorms} on the coordinate ring.
We also introduce $\Xanpm$ for non-affine varieties.
We study the topological properties of $\Xanpm$ and its canonical maps
$\Xanpm \to \Xan$ and $\Xanpm \to X_r$ in \S \ref{sec:Xanpm}.
In comparison with the real spectrum $X_r$, we will see that the
real analytification $\Xanpm$ is in fact homeomorphic to a
subspace of $X_r$ which can be described entirely in terms of orders,
consisting of the so-called relatively archimedean points.

In \S\ref{sec:SA} we study the real analytification $\Sanpm$ for a
semialgebraic set~$S$.
In \S\ref{sec:hyperfields} we explain how to tropicalize polynomial inequalities.
To do this, we take a quick excursion into the realm of hyperfields.
Hyperfields are algebraic structures similar to fields, but which allow multivalued addition.
Jun shows in \cite{Jun} that both the Berkovich space $\Xan$ and the real spectrum $X_r$ of a variety
can be seen as the set of points of $X$ over the {\em tropical hyperfield}
and the {\em sign hyperfield} respectively.
We recall his results and show that $\Xanpm$ can similarly be seen as
the set of points of $X$ over the {\em tropical real hyperfield}.
The results by Jun provided some of our inspiration for the definition of $\Xanpm$
and the tropicalization of polynomial inequalities.

In \S\ref{sec:limit} we turn our attention to the real tropicalization map
and prove our two main theorems \ref{thm:fund} and \ref{signed limit} mentioned above.
An important step is Theorem \ref{strong density}, which shows that
any point in $\trop_r(\Sanpm)$ that satisfies an obviously necessary condition
is already in $\trop_r(S)$.  We also prove the existence of a finite
tropical description for semialgebraic sets in each orthant.


\section{Background}
\Label{sec:background}

For a commutative ring $A$, a  subset $P \subset A$ is called an {\em ordering}
of $A$ if $P+P \subset P$, $P \cdot P \subset P$,
$P \cup -P = A$, and $P \cap -P$ is a prime ideal of~$A$. The
{\em support} of $P$ is $\supp(P) := P \cap -P$.  With $P$ we
associate an order relation and a sign function in the natural way:
Write $f>_P0$ and $\sgn_P(f)=+1$ if $-f\notin P$, write $f\ge_P0$ if
$f\in P$ and $\sgn_P(f)=0$ if $f\in\supp(P)$. We can define $<_P$ and
$\le_P$ in the obvious way, and we have $\sgn_P(-f)= -\sgn_P(f)$. We may
refer to the ordering as either $P$ or~$<_P$.

The {\em real spectrum} $\Sper(A)$ of a commutative ring $A$ is the set of
orderings of $A$, with the {\em Harrison topology} given by the subbasis of
open sets $\{P \in\Sper(A) : f >_P 0\}$ for $f \in A$.  See \cite{bcr}
or \cite{ks}, for example. For an affine scheme $X = \Spec(A)$, the
{\em real spectrum} $X_r$ is defined to be $\Sper(A)$. The real
spectrum $X_r$ of an arbitrary scheme $X$ is defined by glueing the
real spectra of open affine subspaces.  A point in $X_r$ is a pair
$(p,P)$ where $p$ is a scheme point of $X$ and $P$ is an ordering
of the residue field $K(p)$ of~$p$.

For example, for $X = \A^1_\R$, the orderings on $A = \R[T]$ are
classified as follows (see, for example,~\cite{bcr} Example 7.1.4):
\begin{align*}
 P_a &= \{f(T) \in \R[T]: f(a) \geq 0\}, \text{ for } a \in \R\\
  P_{a^+} &= \{f(T) \in \R[T]: f \geq 0 \text{ on some interval
  } [a,a+\varepsilon], \varepsilon > 0 \} , \text{ for } a \in \R\\
 P_{a^-} &= \{f(T) \in \R[T]: f \geq 0 \text{ on some interval
  } [a-\varepsilon,a], \varepsilon > 0 \} , \text{ for } a \in \R\\
 P_\infty &= \{f(T) \in \R[T]: f\ge0\text{ on some interval }
  [b,\infty[\text{ with }b\in\R\}\\
 P_{-\infty} &= \{f(T) \in \R[T]: f\ge0\text{ on some interval }
  \left]-\infty,b\right]\text{ with }b\in\R\}
\end{align*}

For $P_a$, the support is $\langle T-a \rangle$, and the residue field
is $\R$, which has a unique order.  The other four types have
support $\{0\}$ with residue field $\R(T)$.  The real spectrum
$\Sper\R(T)$ is identified with the subspace of $\Sper\R[T]$ consisting
of orderings of the form $P_{a^+}, P_{a^-}, P_{\infty}$ and $P_{-\infty}$.
We can think of $P_{a^+}$ (resp.\ $P_{a^-}$)  as orders in which $T$ is
infinitesimally larger (resp.\ smaller) than $a$.
The points of the form $P_a$, $P_\infty$,  $P_{-\infty}$ are closed, while
$\ol{\{P_{a^+}\}} = \{P_{a^+},\,P_a\}$ and $\ol{\{P_{a^-}\}} =
\{P_{a^-},P_a\}$.

Let $K$ denote a real closed field with a nontrivial
non-archimedean valuation $\val : K^* \rightarrow \R$, giving a
non-archimedean absolute value $|a|_K=|a|= e^{-\val(a)}$ ($a\in K^*$),
$|0|=0$.  Being real closed, $K$ has a unique ordering. We will always
assume that the absolute value is compatible with this ordering:
$0\le a\le b$ $\To$ $|a|_K\le|b|_K$, or equivalently, $|a|_K>|b|_K$
$\To$ $\sgn(a+b)=\sgn(a)=\sgn(a-b)$.  An example is the Puiseux
series field $K = \R\{\!\{t\}\!\} =\bigcup_{n \geq 0} \R((t^\frac{1}{n}))$.
We will equip $K$ with the topology induced by the order.
The compatibility ensures that this topology is coarser then the one induced
by the non-archimedean absolute value and both topologies agree if the absolute value is non-trivial.

For a $K$-algebra $A$,  a \emph{(non-archimedean multiplicative) $K$-seminorm} on $A$ is a map
$\abs_x \colon A \rightarrow \R_{\ge0}$ such that $|a|_x = |a|_K$ for
$a\in K$ and $|fg|_x = |f|_x \cdot |g|_x$, and
$|f+g|_x \le\max(|f|_x,|g|_x)$ ($f,g\in A$).
For this text we only work with non-archimedean multiplicative $K$-seminorms and simply write  ``$K$-seminorm''.
The {\em support}
of the seminorm $\abs_x$ is the prime ideal $\supp(\abs_x)= \{f \in
A : |f|_x = 0\}$.  An absolute value is a seminorm whose support
is $\{0\}$.

The {\em Berkovich spectrum} $\cM(A)$ is the set of $K$-seminorms
on $A$ with the coarsest topology that makes the map $\abs_x \mapsto
|f|_x$ continuous for every $f \in A$.  See~\cite{Berkovich}.
A subbasis of open sets is given by the $\{\abs_x \in \cM(A) \colon
r < |f|_x <s\}$ where $r<s$ are real numbers and $f \in A$.
For an affine $K$-scheme $X = \Spec(A)$, the analytification
$X^{\an}$ is defined to be $\cM(A)$.  The analytification $X^{\an}$
of an arbitrary $K$-scheme $X$ is defined by glueing the spectra of
open affine subschemes.  A point in $X^{\an}$ is a tuple $(p, \abs)$
where $p$ is a point in the scheme $X$ and $\abs$ is an absolute
value on the residue field $K(p)$ which is compatible with
$\abs_K$ on $K$.  See \cite{Baker} for a visualization of
$\A^{1,{\an}}$, which will be discussed in Example~\ref{ex:line}.

Given any non-archimedean valued field $(L,\abs_L)$, we denote by
$k_L$ the residue field of the valuation ring of $L$ and by
$\ol f\in k_L$ the reduction of $f \in L$ with $\vert f\vert\le1$.
We also denote by $\Gamma_L := \vert L^* \vert_L$ the value group
of $L$, a multiplicative subgroup of $\R_{>0}$.
Let $P$ be an ordering of $L$ that is compatible with $\abs_L$. Then
$P$ induces an ordering $\ol P$ of the residue field $k_L$ by
$$\ol P\>=\>\{\ol a\colon a\in P,\ |a|_L\le1\}.$$
Conversely, the Baer-Krull theorem tells which orderings $P$ induce
the same ordering $\ol P$ on $k_L$:

\begin{thm}[Baer-Krull Theorem] \Label{Baer-Krull-Theorem}
Let $(L,\abs_L)$ be a non-archimedean valued field with value
group $\Gamma_L$. Let $(f_i)_{i \in I}$ be a family of elements in
$L^*$ such that $(\vert f_i \vert_L)_{i \in I}$ is a $\Z/2\Z$-basis
of $\Gamma_L / 2\Gamma_L$. Given any ordering $Q$ of $k_L$ and any
tuple $(\epsilon_i)_{i\in I}$ in $\{\pm1\}^I$, there is a unique
ordering $P$ of $L$ that is compatible with $\abs_L$ and satisfies
$\ol P=Q$ and $\sgn_P(f_i)=\epsilon_i$ ($i\in I$). In short, the map
\begin{align*}
\{ P\in\Sper(L)\colon P\text{ is compatible with } \abs_L \}
  & \ \to \{-1, 1\}^I \times \Sper(k_L),  \\
P & \ \mapsto \bigl((\sgn_P(f_i))_{i\in I},\>\ol P\bigr)
\end{align*}
is a bijection.
\end{thm}

It can easily be shown that an order on a non-archimedean valued
field $L$ that is compatible with $\abs_L$ extends in a unique way
to the completion $\hat L$ of $L$.

\begin{dfn}
Let $X = \Spec(A)$ be an affine $K$-variety. 
A semialgebraic set $S \subset X(K)$ is a finite boolean combination
of sets of the form
$\{x \in X(K) : f(x) > 0\}$, where $f \in A$.
An explicit expression for $S$ in this way is called
\emph{semialgebraic description} of $S$.
\end{dfn}


\section{The Real Analytification \texorpdfstring{$\Xanpm$}{Xran}}
\Label{sec:Xanpm}

Let $K$ be a real closed field with an order-compatible non-archimedean
absolute value.
We will assume that this absolute value is non-trivial in Sections \ref{Relation with the real spectrum} and \ref{Paths}.
Let $X$ be a  variety over $K$, by which we always mean reduced, irreducible, separated scheme of finite type.  For a
point $p$ in the scheme $X$, let $K(p)$ denote its residue field.

\begin{dfn}\Label{Xanpm}
The {\em real analytification}
of $X$ is the set $\Xanpm$ consisting of all triples $x=(p_x,\,\abs_x,\,<_x)$ where $p_x \in X$, $\abs_x$ is an absolute value on
$K(p_x)$ extending $\abs_K$, and $<_x$ is an order on $K(p_x)$
compatible with $\abs_x$.
We equip $\Xanpm$ with the coarsest topology such that the
support map
\[\Xanpm \to X,\quad x \mapsto\supp(x):=p_x\]
is continuous and the map
$$\supp^{-1}(U) \to \R,\quad x \mapsto \sgn_{x} (f) \cdot \vert f
  \vert_x\ =: \vert f \vert^\sgn_x $$
is continuous for every open $U \subset X$ and every regular functions $f$ on $U$.
\end{dfn}

\begin{rem}
Note that the real analytification $\Xanpm$ is functorial in $X$
in the obvious way.
\end{rem}

For an affine $K$-variety $X = \Spec(A)$, the real analytification
$\Xanpm$ has a description as the space of signed seminorms, which
combines the constructions of $\Sper(A)$ and $\cM(A)$, as we will
now see.

\begin{dfn} \Label{dfn signed seminorm}
Let $A$ be a $K$-algebra.
A \emph{signed $K$-seminorm} on $A$ is a map $\abs^\sgn \colon A \to \R$
satisfying
\begin{enumerate}
\item
$\vert f \vert^\sgn = \sgn(f) \cdot \vert f \vert_K$ if $f \in K$,
\item
$\vert f \cdot g \vert^\sgn = \vert f \vert^\sgn \cdot \vert g \vert^\sgn$,
\item
$\min(\vert f \vert^\sgn, \vert g \vert^\sgn)\>\le\>\vert f + g \vert^\sgn
\>\le\>\max(\vert f \vert^\sgn, \vert g \vert^\sgn)$
\end{enumerate}
for all $f,g\in A$. The {\em real Berkovich spectrum} $\cM_r(A)$
is the set of all signed $K$-seminorms on $A$, with the coarsest
topology that makes the map
\[ \cM_r(A) \rightarrow \R, \quad \abs^\sgn \mapsto \vert f \vert^\sgn\]
continuous for every $f$ in $A$.
\end{dfn}

\begin{prop}
Let $A$ be a finitely generated $K$-algebra and $X = \Spec(A)$.
Then
\begin{align*}
\Xanpm = \cM_r(A).
\end{align*}
\end{prop}

\begin{proof}
Given $x = (p_x, \abs_x, <_x) \in \Xanpm$ we get a
signed seminorm $\abs^\sgn_x$ on $A$ by
$$|f|_x^\sgn\>:=\>\sgn_x(f_x)\cdot|f_x|_x$$
where $f_x = f + p_x \in  K(p_x)$.

Conversely, given $\abs^\sgn_x \in \cM_r(A)$,
we get a $K$-seminorm $\abs_x $ on $A$ defined by $\vert f \vert_x := \big \vert \vert f
\vert_x^\sgn \big \vert$, and an ordering $\{f \in A : |f|_x^\sgn \geq
0\}$, where the support of the seminorm and the support of the ordering coincide.

These procedures are mutual inverses, and it is easy to see that the topologies agree.
\end{proof}

\begin{prop} \Label{Completion does not matter}
Let $K$ be a real closed field with a compatible non-archimedean absolute value.
Then the completion $\Khat$ of $K$ is also a real closed field with a compatible absolute value.
Further, if $X$ is a variety over $K$, then $(X \otimes \Khat)_r^{\an}$ is homeomorphic to $\Xanpm$.
\end{prop}

\begin{proof}
Let $K'$ be the algebraic closure of $K$ and $\hat{K'}$ its completion.
Then $\hat{K'}$ is algebraically closed and $2 = [K' : K] \geq [\hat{K'}: \Khat]$.
Since the order on $K$ extends to the order on $\Khat$, we have
that $\Khat$ is real, thus real closed.

For the second statement, we may assume that $X$ is affine.
Then the statement is equivalent to showing that any signed seminorm on a $K$-algebra
$A$ extends uniquely to $A \otimes_K \Khat$.

Since a signed seminorm on $K[T_1,\dots,T_n] / \mathfrak{a}$ where
$\mathfrak{a}$ is an ideal is just a signed seminorm on $K[T_1,\dots,T_n]$
whose associated prime contains $\mathfrak{a}$, we may reduce to the case $X = \A^n$.

We use the map $(\A^n_{\Khat})^{\an}_r \to (\A^{n}_K)^{\an}_r$ given by
restriction of signed seminorms.
We construct a section to show that this map is bijective.
We will use the multiindex notation; that is, for a tuple $I = (i_1,\dots,i_n)$, 
we write $T^I := \prod_{j=1}^n T^{i_j}$ and $\vert I \vert = \sum_{j=1}^n i_j$. 
Given a signed seminorm $\abs^\sgn$ on $K[T_1,\dots,T_n]$ and $f = \sum_{I : \vert I \vert \leq d} a_I T^I \in \Khat [T_1,\dots,T_n]$
we pick sequences $b_{I,j} \to a_I$ and define $f_j := \sum_{I : \vert I \vert \leq d} b_{I,j} T^I$ and 
$\vert f \vert^\sgn = \lim_{j \to \infty} \vert f_j \vert^\sgn$.
It is easy to check that this is well defined and indeed defines an extension of $\abs^{\sgn}$.

We now identify these two sets along this bijection and consider the two topologies,
one defined by $(\A^n_K)^{\an}_r$ and the other one by $(\A^n_{\Khat})^{\an}_r$.
Clearly the one defined by $\Khat$ is finer then the one defined by $K$.

To see that the topologies agree, fix $f = \sum_{I : \vert I \vert \leq d} a_I T^I \in \Khat[T_1,\dots,T_n]$ and $c \in \R$.
Take $\abs_y^{\sgn} \in (\A^n_K)^{\an}_r$ with $\vert f \vert^\sgn_y > c$
and fix $s > \vert T_i \vert_y$ for all $i$.
Take $g = \sum_{I : \vert I \vert \leq d} b_I T^I \in K[T_1,\dots,T_n]$ of same degree as $f$ such that
$\vert a_I - b_I \vert < s^{-1} \cdot \deg(f)^{-1} \cdot c$.
Then for all $\abs^\sgn \in (\A^n_{\Khat})^{\an}_r$ we have
$\vert f - g \vert < c$.
Then $\vert g \vert^\sgn > c$ implies $\vert f \vert^\sgn > c$ and we have
shown that $\{ \abs^\sgn : \vert f \vert^\sgn > c\}$
contains the set $\{ \abs^\sgn : \vert g \vert^\sgn > c\}$
that contains $\abs_y^\sgn$.
Since the latter set is open in the topology of $(\A^n_K)^{\an}_r$,
this shows that the two topologies agree.
\end{proof}

\begin{prop}
Let $X$ be a variety over $K$ (in particular separated).
Then $\Xanpm$ is a Hausdorff space.
\end{prop}

\begin{proof}
We may assume that $K$ is complete by Proposition \ref{Completion does not matter}.
Let $x \neq y$ be two points in $\Xanpm$.
It is sufficient to construct a continuous map from an open subset 
of $\Xanpm$ that contains $x$ and $y$ to a Hausdorff space that separates $x$ and $y$.

If there exists an open affine subscheme of $U = \Spec(A)$ of $X$ that contains both $p_x$ and $p_y$,
we are reduced to the affine case, since $U_r^{\an}$ is an open subset of $\Xanpm$.
Since $x \neq y$, there exists $f \in A$ such that $\vert f \vert_x^\sgn \neq \vert f \vert_y^\sgn$
and the map $\cM_r(A) \to \R, \; \abs^\sgn \mapsto \vert f \vert^\sgn$ separates $x$ and $y$.

If there does not exist such an open affine $U$, we use the canonical map
$\psi \colon \Xanpm \to \Xan$.
Since $p_x \neq p_y$, certainly $\psi(x) \neq \psi(y)$.
Since $X$ is separated, $\Xan$ is Hausdorff by \cite[Theorem 3.4.8]{Berkovich} and
the proposition follows.
\end{proof}

\subsection{Tropicalization of \texorpdfstring{$\Xanpm$}{Xran}}
Let $X$ be an affine $K$-variety. Given a family of regular functions
$\scrF = (f_1,\dots,f_n)$ on $X$, we have a natural real tropicalization map
$$\trop_{r,\scrF} \colon \Xanpm \to \R^n,\quad
x \mapsto (\vert f_i \vert^\sgn_x )_{i=1,\dots,n}$$
By construction, this map is continuous.

\begin{lem} \Label{trop proper}
If $f_1,\dots,f_n$ generate the coordinate ring $A$ of $X$, the
map $\trop_{r,\scrF}$ is
a proper map of topological spaces.
\end{lem}

\begin{proof}
Let $B$ be a compact subset of $\R^n$.
Since $B$ is bounded, writing $B' = (\trop_{r,\scrF})^{-1}(B)$,
each $f_i$ is bounded on $B'$. Since the $f_i$'s generate $A$, by the
ultrametric triangle inequality $\{\vert f \vert_x : x \in B'\}$
is bounded for each $f \in A$.
By construction, the map
$$\Xanpm \to \prod_{f \in A} \R,\quad
x \mapsto  \bigl(\vert f \vert^\sgn_x \bigr)_{f \in A}$$
is a homeomorphic embedding.
Its image is closed, since it is defined by closed conditions
\ref{dfn signed seminorm} (1), (2) and (3).
Thus $B'$ is a closed subset of $\prod_{f \in A} \R$.

Since all $f \in A$ are bounded on $B'$, the image of $B'$ is contained
in a product of bounded closed intervals,
thus contained in a compact subset by Tychonoff's theorem.
Since $B'$ is closed, it is compact.
\end{proof}

Again for a family $\scrF = (f_1,\dots,f_n)$ of elements of $A$, there is a tropicalization map
on $X^{\an}$ given by
\[\trop_{\scrF} \colon X^{\an} \rightarrow \R^n, \quad \abs_x \mapsto
(|f_i|_x)_{i=1,\dots,n}.\]

\begin{lem} \Label{forget trop proper}
The following map is proper:
\[\trop_{r,\scrF}(\Xanpm) \to
\trop_{\scrF}(\Xan), \quad (a_1,\dots,a_n)\mapsto
(|a_1|,\dots,|a_n|).\]
\end{lem}

\begin{proof}
The left hand set is a closed subset of $\R^n$ by Lemma
\ref{trop proper}. The map $\R^n \to \R^n_{\geq 0}$,
$a \mapsto \vert a \vert$ is proper, thus so is its restriction
to the closed subset $\trop_{r,\scrF}(\Xanpm)$ of $\R^n$.
\end{proof}

\subsection{Relation with the Berkovich Analytification}

Let $X$ be a $K$-variety.
In this section we study $\Xanpm$ via the canonical map
\begin{align*}
\varphi \colon \Xanpm \to \Xan, \quad (p_x, \abs_x, \le_x) \mapsto (p_x, \abs_x).
\end{align*}

\begin{lem} \Label{forget an proper}
The map $\varphi$ is a proper map of topological spaces.
\end{lem}

\begin{proof}
We may assume that $X = \Spec(A)$ is affine.
Let $\scrF = (f_1,\dots,f_n)$ be a family of elements of $A$ that generate $A$.
Let us consider the diagram
\begin{align*}
\begin{xy}
\xymatrix{
\Xanpm \ar[rr] \ar[d] && \trop_{r,\scrF}(X) \ar[d] \\
\Xan      \ar[rr]         &&   \trop_{\scrF}(X).
}
\end{xy}
\end{align*}
Both of the horizontal maps and the right vertical map are proper by
Lemma \ref{trop proper} and \ref{forget trop proper},
so chasing a compact subset of $\Xan$ through the diagram shows that
the left vertical map is also proper.
\end{proof}

\begin{cor}\Label{proper implies compact}
If $X$ is a proper $K$-variety, then $\Xanpm$ is compact.
\end{cor}

\begin{proof}
We may assume that $K$ is complete by Proposition \ref{Completion does not matter}. 
Then this follows from Lemma \ref{forget an proper} since $\Xan$ is
compact by \cite[Theorem 3.4.8]{Berkovich}.
\end{proof}

\begin{prop}
\Label{prop:fiber}
Let $x = (p_x, \abs_x) \in \Xan$, let $\Gamma_x$
be the value group and $k(x)$ the residue field of $\abs_x$ on
$K(p_x)$. Then $\varphi^{-1}(x)$ is homeomorphic to $D\times\Sper k(x)$,
where $D$ is a discrete set of cardinality $|\Gamma_x/2 \Gamma_x|$.
\end{prop}

\begin{proof}
This is a direct consequence of the Baer-Krull Theorem (Theorem \ref{Baer-Krull-Theorem}).
\end{proof}

\begin{example}[The affine line]
\Label{ex:line}
We will now describe the affine line over $K$,
where $K$ is real closed and complete with respect to a non-archimedean absolute value.
We will describe $\A_r^{1, \an}$ by considering the map $\A_r^{1, \an}
\to \A^{1, \an}$ and looking at the fibers.

We view $\A^{1, \an}$ as the set of pairs $x = (p_x, \abs_x)$
where $p_x$ is a point in the scheme $\A^1$ and $\abs_x$ is an absolute
value on the residue field $K(p_x)$.
Let $k(x)$ be the residue field of the valuation ring of $K(p_x)$ with respect to $\abs_x$.
We denote by $\Hscr(x)$ the completion of $K(p_x)$ with respect to
$\abs_x$.  Its residue field with respect to $\abs_x$ is also $k(x)$.
Let $K' = K(\sqrt{-1})$ be the algebraic closure of $K$.
The points in $\A^{1,\an}$ are classified by Berkovich into four types
as follows \cite[p.~17]{Berkovich}.

\begin{description}
\item[Type I] $p_x$ is a closed point of $\A^1$, and $|f|_x =
  |f(a)|_{K'}$ for some $a \in K'$.  If $a \in K$, then $p_x = \langle
  T - a\rangle$ and $K(p_x) = K$. Otherwise $p_x = \langle
  (T-a)(T-\overline{a})\rangle$, and $K(p_x)=K'$.
\item[Type II] $p_x$ is the generic point on $\A^1$, so $K(p_x) = K(T)$,
the field of rational functions in one variable $T$ over $K$.
The field $k(x)$ is of transcendence degree~$1$ over $k_K$ and
$\vert \Hscr(x) \vert_{x} = \vert K \vert_K$.
\item[Type III]  $p_x$ is the generic point on $\A^1$, $k(x)$ is algebraic over $k_K$ and
$\vert \Hscr(x) \vert_{x} / \vert K \vert_K$ is a free abelian group of rank~$1$.
\item[Type IV] $p_x$ is the generic point on $\A^1$,
$k(x)$ is algebraic over $k_K$ and
$\vert \Hscr(x) \vert_{x} = \vert K \vert_K$.
\end{description}

We say that a point $x = (p_x, \abs_x)$ in $\A^{1,\an}$ is {\em real}
if there is an ordering of $K(p_x)$ compatible with $\abs_x$, that is,
its fiber under the map $\A_r^{1, \an} \to \A^{1, \an}$ is non-empty.
Using Proposition~\ref{prop:fiber} we can describe the fibers of $x \in \A^{1,\an}$ under the map $\A_r^{1, \an} \to \A^{1, \an}$ as follows.
\begin{center}
\begin{tabular}{l|c|c|c}
Type & $k(x)$ & $|\Gamma_x / 2\Gamma_x|$ & fiber \\
\hline  \hline
Type I, real & $k_K$ & 1 & one point, since $k_K$ has a uniqe ordering \\ \hline
Type I, non-real & $k_{K'}$ & 1 & $\varnothing$ \\ \hline
Type II, real & $k_K(T)$ & 1  & homeomorphic to $\Sper k_K(T)$ \\ \hline
Type II, non-real& $k_{K'}(T)$ & 1 & $\varnothing$ \\ \hline
Type III, real& $k_K$ & 2 & two points with discrete topology \\ \hline
Type III, non-real& $k_{K'}$ & 2 & $\varnothing$ \\ \hline
Type IV, real& $k_K$ & 1 & one point \\ \hline
Type IV, non-real& $k_{K'}$ & 1 & $\varnothing$ \\ \hline
\end{tabular}
\end{center}

When $k_K = \R$, such as when $K = \R\{\!\{t\}\!\}$, the real spectrum $\Sper k_K(T)$ is described explicitly in Section~\ref{sec:background}.
In general it can be described using Dedekind cuts.

For a point $x = (p_x, \abs_x)$ of type II or III, there exists a unique closed disc $D(a,r) = \{z \in K : |z - a| \leq r\}$ 
where $a$ is in $K'$ and $r$ is a positive real number, such that $\abs_x = \sup_{c \in D(a,r)} \vert f (c) \vert_{K'}$.  
The point $x$ is of type II if $r \in | K |_K$ and type III otherwise. 
It is real if and only if the disk $D(a,r)$ contains a point in $K$, which means $|\im(a)| \leq r$, 
where $\im(a)$ denotes the imaginary part of $a$.  
Type IV points can be constructed from nested sequences of discs, and the real ones are obtained as sequences of discs with real points.
\end{example}

\subsection{Relation with the Real Spectrum} \Label{Relation with the real spectrum}

In this section, we assume that $K$ is real closed and
that $\abs_K$ is non-trivial.
Then $\vert K \vert_K$ is necessarily dense in $\R_{\geq 0}$.
Further let $X$ be a variety over $K$. 

We first prove the following crucial lemma, which says that the
absolute value of an extension field $L \supset K$ is determined by
the position of its elements relative to $K$.

\begin{lem}  \Label{order implies abs}
Let $L/K$ be an extension of ordered fields with a compatible
non-archimedean absolute values.
Let $f \in L$ and
\begin{align*}
R(f) = \{ r \in  K_{\geq 0} : r \leq \sgn(f) \cdot f \}
\end{align*}
Then $\vert f \vert = \sup \{ \vert r \vert : r \in R(f) \}$.
\end{lem}

\begin{proof}
We may assume $f \neq 0$.
Replacing $f$ by $\sgn(f) \cdot f > 0$ we may assume $f > 0$.
Thus if $0 \leq r \leq \sgn(f) \cdot f$, we have $\vert r \vert \leq \vert f \vert$,
by the condition of compatibility between order and absolute value.
Conversely if $\vert r \vert > \vert f \vert $, then $ r  > f $.
Now the result follows since $\vert K \vert$ is dense in $\R_{\geq 0}$.
\end{proof}

\begin{dfn}
Let $X_r^\mx$ be the set of closed points of $X_r$, and let
$X_r^\arch\subset X_r$ consist of those points $(p,P)\in X_r$ whose
ordered residue field is relatively archimedean over $K$. By this we
mean that, for every $b\in K(p)$, there exists $a\in K$ with
$a-b\in P$.
\end{dfn}

\begin{lem}
We have $X_r^{\arch} \subset X_r^\mx$.
\end{lem}
\begin{proof}
This follows from the description of specializations in the real spectrum,
see \cite[III.7, in particular Kor.~5]{ks}.
\end{proof}

\begin{prop} \Label{Xan to Xarch}
The image of the canonical map $\psi \colon \Xanpm \to X_r$
is $X_r^{\arch}$.
\end{prop}

\begin{proof}
Let $O_K=\{a\in K\colon|a|\le1\}$ be the given valuation ring of $K$.
We may assume that $X=\Spec(A)$ is affine, so $A$ is a finitely
generated $K$-algebra. 
The image of $\Xanpm$ in $X_r=\Sper(A)$ consists of all $(p,P)\in
X_r$ for which there is a maximal proper {\em convex} subring $B$ of $K(p)$
with $K\cap B=O_K$.  (A subring $B$ of $K(p)$ is called {\em convex} if for any $b \in B$
 and $a \in K(p)$ with $|a| < |b|$, we have $a \in B$.)
Given $(p,P)\in X_r$, note that there always is
a largest proper convex subring $B$ of $K(p)$ (since $K(p)/K$ has
finite transcendence degree), and that $O_K\subset B$.
By the previous remark, $(p,P)\in\Xanpm$ if and only if
$K\cap B=O_K$. This in turn holds if and only if $K$ is not contained
in $B$, or equivalently, if and only if $K(p)/K$ is relatively
archimedean.
\end{proof}

\begin{thm} \Label{forget real injective} \Label{xanpm=xrarch}
The canonical map $\psi \colon \Xanpm \to X_r^{\arch}$
is a homeomorphism.
\end{thm}

\begin{proof}
This map is injective by Lemma \ref{order implies abs} and surjective by Proposition \ref{Xan to Xarch}.

Since proving the statement for a variety $X$ clearly implies the statement for all open subvarieties,
by Nagata's compactification theorem (see \cite[Theorem 4.1]{Vojta})
 we may assume that $X$ is proper.
Then $\Xanpm$ is compact by Corollary \ref{proper implies compact} and $X_r^\arch$ is Hausdorff.

We showed that $\psi$ is a continuous bijection from a compact space to a Hausdorff space,
thus necessarily a homeomorphism.
\end{proof}

For the rest of this section, we view $\Xanpm$ as a topological
subspace of $X_r$ by means of the topological embedding
$\psi$. Every $K$-valued point in $X(K)$ defines
a point in $\Xanpm$ in a canonical way.

It is well known that $X(K)$ is dense in $X_r$, so we have the following corollary.
\begin{cor} \Label{rational points dense}
We have the following inclusions of dense subspaces of $X_r$:
\[
X(K) \subset \Xanpm = X_r^\arch \subset X_r^\mx \subset X_r.
\]
\end{cor}

If $X$ is proper then $\Xanpm$ is compact (Corollary \ref{proper implies compact}), so we have the following,
since $X_r^{\max}$ is Hausdorff.

\begin{cor} \Label{cor density II}
If $X$ is proper, then $X_r^\mx=\Xanpm$.
\end{cor}


\subsection{Paths} \Label{Paths}

In this section we again assume that the absolute value on $K$ is non-trivial.

While $\Xan$ is path connected if $X$ is connected \cite[Theorem 3.2.1]{Berkovich},
there is no corresponding property of $\Xanpm$.
Indeed, we do not have any non-constant paths at all.

\begin{thm} \Label{no paths}
Let $F \colon [0,1] \to \Xanpm$ be a continuous map.
Then $F$ is constant.
\end{thm}

We first prove the following lemma:

\begin{lem} \Label{fiber is A1}
Let $X = \Spec(A)$ be an affine scheme, let $x\in \Xanpm$, and let
$L$ be the real closure of the ordered quotient field of $A/p_x$.
Denoting by $\pi \colon (\A^1 \times X)_r^{\an} \to \Xanpm$
the projection, the canonical map
\begin{align*}
(\A^1_L)_r^{\an} \to \pi^{-1}(x)
\subset (\A^1 \times X)_r^{\an}
\end{align*}
is a homeomorphism.
\end{lem}

\begin{proof}
Let us denote by $\pi_r \colon (\A^1 \times X)_r \to X_r$ the canonical map on real spectra.
Identifying $\Xanpm$ with the subset $X_r^{\arch}$ of $X_r$,
we have to show that $\A^{1, \arch}_{L, r} \to \pi_r^{-1}(x) \cap
(\A^1 \times X)_r^{\arch}$
is a homeomorphism.
Since $\A^1_{L, r} \to \pi_r^{-1}(x)$ is a homeomorphism by \cite[Prop.\ 4.3]{cr} we have to show that an order on $A[T]$
that extend the order $P_x$ is relatively archimedean over $K$ if and only if its extension to $L[T]$
is relatively archimedean over $L$.
This is implied by the fact that $L$ is relatively archimedean over $K$, which holds since $x \in X_r^{\arch}$.
\end{proof}

\begin{proof}[Proof of Theorem \ref{no paths}]
We treat the case $X = \A^1$ first.
We may assume that $K$ is complete by Lemma \ref{Completion does not matter}.
Assume that $F$ is not constant.
By considering the composition of $F$ with the canonical map $\varphi \colon \Xanpm \to \Xan$,
after possibly shrinking $F$ we may assume that $F([0,1]) \cap X(K) = \emptyset$.
Pick two points $\abs_1^\sgn \neq \abs_2^\sgn \in F([0,1])$.
Since these define different points in $\A^1_r$, there exists $f \in K[T]$ such that
$\sgn_1 (f) \neq \sgn_2(f)$.
Then $F^{-1}(V(f)) \cup F^{-1}(V(-f))$ is a cover of $[0,1]$ by disjoint non-empty open subsets, a contradiction.
Note here that since the image of $F$ does not contain any rational points, $f$ does not vanish on the image of $F$.

Now assume the statement is known for $\A^n$.
Let $F \colon [0,1] \to \A_r^{n+1, \an}$.
If we denote by $\pi \colon \A_r^{n+1, \an} \to \A_r^{n, \an}$ the projection,
then $\pi \circ F$ has to be constant by the inductive hypothesis, thus $F([0,1])$ is contained in a fiber of $\pi^{-1}(x)$.
By Lemma \ref{fiber is A1} this fiber is homeomorphic to $\A^{1, \an}_{L,r}$ for an extension $L / K$.
Again by the inductive hypothesis, $F$ has to be constant.

The case of general $X$ follows from the affine case, which in turn follows from the $\A^n$ case.
\end{proof}


\section{Semialgebraic Sets}
\Label{sec:SA}

As before, let $K$ be a real closed field with a non-archimedean absolute value,
and $X = \Spec(A)$ an affine $K$-variety. 
Let $S\subset X(K)$ be a semialgebraic subset.

We denote by $\tilde S$ the constructible subset of $X_r = \Sper A$
associated with the semialgebraic set $S$.  
In particular, if $S= \bigcup_{i=1}^m \{\xi\in X(K) \colon\sgn f_{ij}(\xi)=\epsilon_{ij} \; j=1,\dots,r\}$ is any finite semialgebraic description of $S$ 
(with $f_{ij}\in A$ and $\epsilon_{ij}\in\{-1,0,1\}$), 
then $\tilde S$ has the same description in the real spectrum, i.e.
$$\tilde S\>=\> \bigcup_{i=1}^m \{x \in\Sper A \colon\sgn_x(f_{ij})=\epsilon_{ij}, \; i=1,\dots,r\},$$
and the mapping $S \mapsto \tilde S$ is compatible with boolean combinations
and taking closure and interior. 
This construction does not depend on the choice of the semialgebraic description of $S$~\cite[Proposition~7.2.2]{bcr}.

Recall the we denote by $\psi \colon \Xanpm \to X_r$ the canonical map. 

\begin{dfn}\Label{dfnmanpm}
The {\em real analytification} $\Sanpm$ of $S$ is the preimage of $\tilde S$ under $\psi$:
\[\Sanpm := \psi^{-1}(\tilde S).\]
\end{dfn}

In particular, if $S = X(K)$, this definition gives $\Sanpm=\Xanpm$. 
Note that $\Sanpm$ is a Hausdorff topological space.
It follows directly from the definitions that 
the construction $S \mapsto \Sanpm$ is compatible with boolean combinations.

Assume for the rest of this section that the absolute value on $K$ is non-trivial. 
Then $\psi$ is a homeomorphism onto its image by Theorem \ref{forget real injective} 
and forming $\Sanpm$ is compatible with taking closures and interiors. 

We generalize Corollaries~\ref{rational points dense} and~\ref{cor density II} to
the semialgebraic case, and we prove a converse to the latter:

\begin{prop}\Label{mkdense}
Let $S\subset X(K)$ be a semialgebraic set as above.
\begin{itemize}
\item[(a)]
The inclusions $S\subset\Sanpm=\tilde S^\arch\subset\tilde S^\mx$ of
Hausdorff spaces hold, and $S$ is dense in $\tilde S^\mx$.
\item[(b)]
Equality $\Sanpm=\tilde S^\mx$ holds if and only if
$S$ is semialgebraically compact, and if and only if $\Sanpm$ is compact.
\end{itemize}
\end{prop}

\begin{proof}
(a)
The inclusions are clear, and $S$ is well known to be dense in
$\tilde S$. For (b) we have to show: $\tilde S^\mx\subset\tilde S^\arch$
$\iff$ $S$ is semialgebraically (s.a.) compact. This is well-known,
but for lack of a suitable reference we indicate a proof: If $S$
is s.a.\ compact, then for every $f\in K[X]$ there is $c\in K$ with
$f<c$ on $S$. So from \cite[III.7]{ks} we see that every point in
$\tilde S$ specializes to a point in $\tilde S^\arch$. Conversely,
if $S$ is not s.a.\ compact, there exists a semialgebraic curve in
$S$ without an endpoint in $S$. Associated with it is a point in
$\tilde S^\mx\setminus\tilde S^\arch$.
\end{proof}

It is easy to see that the spaces $\Sanpm$ are functorial with respect to
morphisms of $K$-varieties. In fact they are functorial
also with respect to semialgebraic maps. Recall that a map
$\phi\colon S\to N$ (between semialgebraic subsets of affine
$K$-varieties) is called semialgebraic if $\phi$ is continuous and
its graph is a semialgebraic subset of $S\times N$.

\begin{lem}\Label{safunct}
Let $\phi\colon S\to N$ be a semialgebraic map. There is a unique
continuous map $\phi_*\colon\Sanpm\to N_r^{\an}$ making the diagram
$$\begin{tikzcd}
\Sanpm \arrow[r,"\phi_*"] \arrow[d,swap,"\psi"] &
  N_r^{\an} \arrow[d,"\psi"] \\
\tilde S \arrow[r,"\tilde\phi"] & \tilde N
\end{tikzcd}$$
commute, where $\tilde\phi$ is the map induced by $\phi$ in the
real spectrum.
\end{lem}

\begin{proof}
The map $\tilde\phi$ satisfies $\tilde\phi(\tilde S^\arch)\subset
\tilde N^\arch$, since for every $x\in\tilde S$, $\phi$ induces an
embedding $K(\tilde\phi(x))\to K(x)$ of the ordered residue fields.
Therefore the restriction $\tilde S^\arch\to\tilde N^\arch$ of
$\tilde\phi$ has the desired properties, upon making the
identifications $\Sanpm=\tilde S^\arch$ and $N_r^{\an}=
\tilde N^\arch$. Uniqueness is clear since $N_r^{\an}$ is Hausdorff
and $S$ is dense in $\Sanpm$.
\end{proof}

\begin{thm}\Label{thm:connected}
Let $X$ be an affine $K$-variety, let $S\subset X(K)$ be a
semialgebraic set, with corresponding constructible subset $\tilde S$
of $X_r$. The map $\psi \colon \Xanpm\to X_r$ induces a bijection between the
connected components of $\tilde S$ and those of $\Sanpm$: For any
connected component $C$ of $\tilde S$, the preimage of $C$ in
$\Sanpm$ is connected.
\end{thm}

Note that in turn, the connected components of $\tilde S$ are in
natural bijection with the semialgebraic connected components
of~$S$.

\begin{proof}
It suffices to show that when $S$ is semialgebraically connected,
then $\Sanpm \cong S^\arch$ is connected as well. Let
$U_i\subset X_r$ ($i=1,2$) be open subsets such that
$\tilde S^\arch\subset U_1\cup U_2$ and $\tilde S^\arch\cap
U_1\cap U_2=\emptyset$, and assume $\tilde S^\arch\cap U_i
\ne\emptyset$ for $i=1,2$. Then for $i=1,2$ there exists a
$K$-rational point $P_i$ in $S\cap U_i$. Let $J=[0,1]_K$. Since $S$
is semialgebraically path-connected, there exists a semialgebraic
path $\gamma\colon J\to S$ from $P_1$ to $P_2$. Using Lemma
\ref{safunct}, it follows that $J_r^{\an}$ is disconnected. But $J$
is semialgebraically compact, and so $J_r^{\an}=J_r^\mx$ is
connected as well, a~contradiction.
\end{proof}


\section{Hyperfields and Tropicalization of Inequalities}
\Label{sec:hyperfields}
\subsection{Hyperfields}
\label{sec:hyperfield}

A hyperfield/hyperring $\mathbb{H}$ is a set with a multiplication $\cdot$ and addition $\oplus$, where
addition may be multivalued, that satisfies a set of axioms similar to those for a field/ring.
We first recall some useful hyperfields as introduced by
Viro in \cite{Viro}, where definitions can be found.

\begin{itemize}
\item
The hyperfield of signs $\mathbb{S}$ has multiplicative
group  $(\{\pm 1\}, \cdot)$ with addition where $0$ is the neutral element and $1\oplus 1 = 1, -1 \oplus -1 = -1$ and $-1 \oplus 1 = \mathbb{S}$.

\item
The tropical hyperfield $\T$, with multiplicative notation, has as
multiplicative group $(\R_{> 0},\cdot)$, and the addition is defined by
\begin{align*}
a \oplus b =
\begin{cases}
\max (a,b) \text{ if } a \neq b\\
[0, a] \text{ if } a = b.
\end{cases}
\end{align*}

\item
The real tropical hyperfield $\R\T$ has multiplicative group $(\R^*, \cdot)$ with addition
\begin{align*}
a \oplus b =
\begin{cases}
a \text{ if } \vert a \vert > \vert b \vert  \\
b \text{ if } \vert a \vert < \vert b \vert  \\
a \text{ if } a = b \\
[a, b] \text{ if } a = -b \leq 0 \\
[b, a] \text{ if } a = -b \geq 0.
\end{cases}
\end{align*}
\end{itemize}

A morphism of hyperfields $\varphi \colon \mathbb{H}_1 \to
\mathbb{H}_2$ is a map which induces a morphism of multiplicative groups and satisfies $\varphi(x \oplus y) \subset \varphi(x) \oplus \varphi(y)$.
It is a direct consequence of the definitions that a
seminorm on a ring $A$ is the same as a morphism of hyperrings
$A \to \T$.
Similarly, a signed seminorm on $A$ is an morphism $A \to \R\T$
and an ordering on $A$ is a morphism $A \to \mathbb{S}$.

Let $\mathbb{H}$ be a hyperfield with a topology and
let $K$ be a field with a morphism $K \to \mathbb{H}$.
Let $A$ be a $K$-algebra.
Denote by $\Hom_K(A, \mathbb{H})$ the set of homomorphisms of hyperrings $\varphi \colon A \to \mathbb{H}$
which extend the given morphism on $K$.
We endow it with the coarsest topology such that for all $f \in A$ the map 
$\Hom_K(A, \mathbb{H}) \to \mathbb{H}, \varphi \mapsto \varphi(f)$ is continuous.

Jun showed that we have canonical homeomorphisms 
\begin{align*}
\Hom_K(A, \T) = \mathcal{M}(A) \quad  \text{and} \quad  \Hom_K(A, \mathbb{S}) = \Sper(A),
\end{align*}
when $\T$ has the topology of $\R_{\geq 0}$ and $\mathbb{S}$ has  the topology of $\R / \R_{>0}$
\cite[Proposition 5.8 \& Lemma 6.11]{Jun}.

From this point of view, one can furthermore canonically see the tropicalization map.
Given a family $\scrF = (f_1,\dots,f_n)$ we have
\begin{align*}
\trop_{\scrF} \colon \Hom_K(A, \T) &\to \T^n, \\
\varphi &\mapsto (\varphi(f_1),\dots,\varphi(f_n)).
\end{align*}

Our definition of $\Xanpm$ is inspired by Jun's work, and we have
\begin{align*}
\mathcal{M}_r(A) = \Hom_K(A, \R \T) \quad \text{and} \quad
\trop_{r,\scrF}\colon \Hom_K(A, \R \T) &\to \R\T^n, \\
\varphi &\mapsto   (\varphi(f_1),\dots,\varphi(f_n)).
\end{align*}

Jun also globalizes these constructions and defines the set of $\mathbb{H}$-rational points $X_K(\mathbb{H})$ for not
necessarily affine $K$-varieties $X$.
He also shows the analogous results
\begin{align*}
\Xan = X_K(\T) \text{ and } X_r = X_K(\mathbb{S}).
\end{align*}
Similarly one can show
\begin{align*}
\Xanpm = X_K(\R\T).
\end{align*}

\subsection{Real Tropicalization of \texorpdfstring{$K^n$}{Kn} and Tropicalizing Polynomial Inequalities}
\Label{sec:ineq}

For a hyperfield $\mathbb{H}$, a polynomial in the variables $x_1,\dots,x_n$ over $\mathbb{H}$ is a formal expression
$F = \bigoplus_{d_1,\dots,d_n \in \Z} a_{d_1,\dots,d_n} x_1^{d_1} \cdots x_n^{d_n}$ 
where all but finitely many coefficients $a_{d_1,\dots,d_n}$ are zero.
For $c = (c_1,\dots,c_n) \in \mathbb{H}^n$ we have 
$F(c) = \bigoplus_{d_1,\dots,d_n} a_{d_i,\dots,d_n} c_1^{d_i} \cdots c_n^{d_n} \subset \mathbb{H}$.

Given a tropical polynomial $F$ in $n$ variables over the real tropical
hyperfield $\R\T$, we define the tropical
semialgebraic sets in $\R\T^n$ defined by $F = 0$, $F \geq 0$, and $F
> 0$ to be, respectively, the
sets
\begin{align*}
\{F = 0\} &:=  \{(x_1,\dots,x_n)\in\R\T^n : F(x_1,\dots,x_n) \ni 0 \} \\
\{F \geq 0\} &:=  \{(x_1,\dots,x_n)\in\R\T^n : F(x_1,\dots,x_n) \text{ contains a  nonnegative number}\} \\
\{F > 0\} &:=  \{(x_1,\dots,x_n)\in\R\T^n : F(x_1,\dots,x_n) \text{ is a positive number}\}.
\end{align*}
Recall that the tropical addition may be multivalued, and
$F(x_1,\dots,x_n)$ is either a point or an interval of the form $[-a,
a]$.

For a polynomial $f = \sum a_{d_1,\dots,d_n} x_1^{d_1} \dots x_n^{d_n} \in K[x_1,\dots,x_n]$, let $\trop_r(f)$ be the
polynomial $\bigoplus \vert a_{d_1,\dots,d_n} \vert^\sgn x_1^{d_1} \dots x_n^{d_n}$ over $\R\T$ obtained by replacing $+$ with $\oplus$ and
each coefficient with its signed seminorm.  
For any $a \in K^n$, we
have $|f(a)|^\sgn \in \trop_r(f)(|a|^\sgn)$ because
the signed seminorm is a morphism of hyperfields from $K$ to $\R\T$.
In particular, if a point $a \in K^n$ satisfies $f \geq 0$, then $|a|^\sgn$
satisfies $\trop_r(f)(|a|^\sgn) \geq 0$.  On the other hand, if a point $a \in
K^n$ satisfies $\trop_r(f)(|a|^\sgn) > 0$,
then $f(a) > 0$.

In general, even if $K \to \R\T$ is surjective, it is not the case that
\begin{align*}
\trop_r\left(\{x \in K^n : f(x) \geq 0\}\right) = \{\trop_r(f)(x)\geq 0\}.
\end{align*}
For example, let $S = \{(x,y) \in K^2 : (x-2)^2 + (y-2)^2 \leq 1\}$.
Then $\trop_r(S)$ consists of a single point $(1,1)$.  However
\[\trop_r( (x-2)^2 + (y-2)^2-1) = X^2 \oplus Y^2 \oplus -X \oplus -Y
\oplus 1 \leq 0\] cuts out two line segments which join $(1,1)$ to $(-1,1)$
and $(1,-1)$ respectively.
For more examples, see
Figures~\ref{fig:signedtrop},\ref{fig:rectangle}, and~\ref{fig:trop_intersection}.
We will see in Theorem~\ref{thm:fund} that tropicalizing {\em
  all} polynomial inequalities satisfied by a semialgebraic set $S$
cuts out $\trop_r{S}$, if $K \to \T\R$ is surjective.

\begin{figure}
\begin{tabular}{cc}
\begin{minipage}{0.5\textwidth}
\begin{center}
\includegraphics[width=0.8\linewidth]{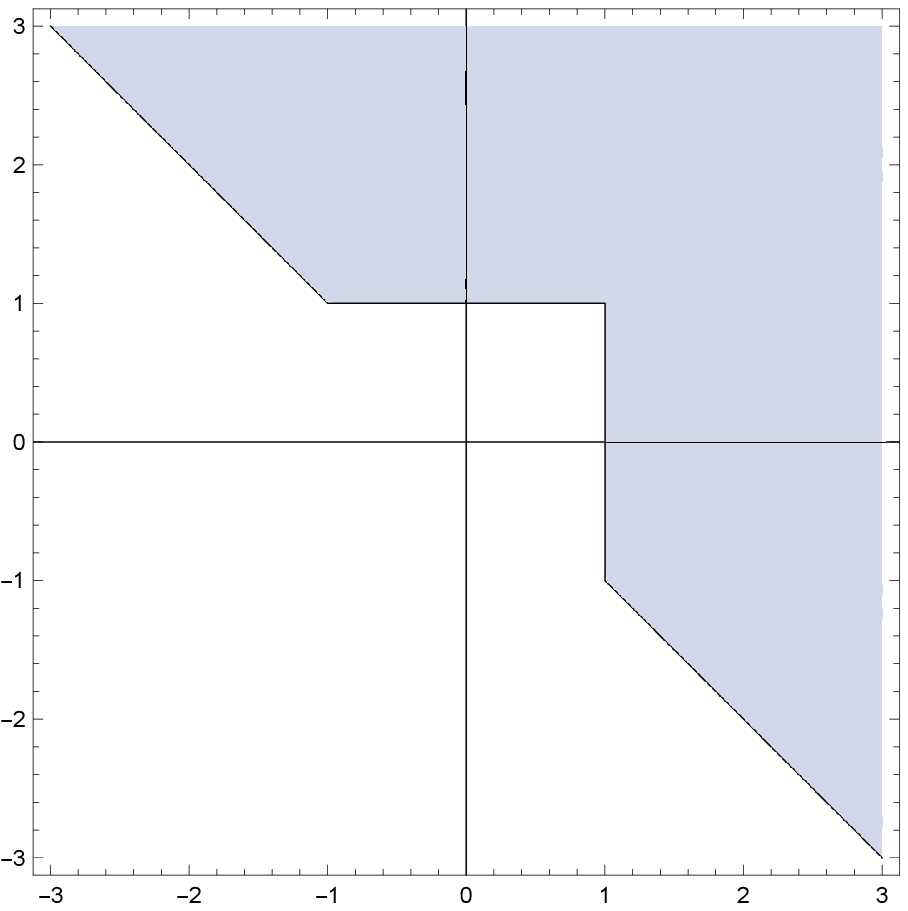}
\end{center}
The subset of $\R\T^2$ defined by \[X \oplus Y \oplus -1 \geq 0.\]  It
coincides with the real tropicalization of, e.g. \[\{(x,y) \in K^2 :
  2x+3y-5\geq 0\}.\]
\end{minipage} &
\begin{minipage}{0.5\textwidth}
\begin{center}
\includegraphics[width=0.8\linewidth]{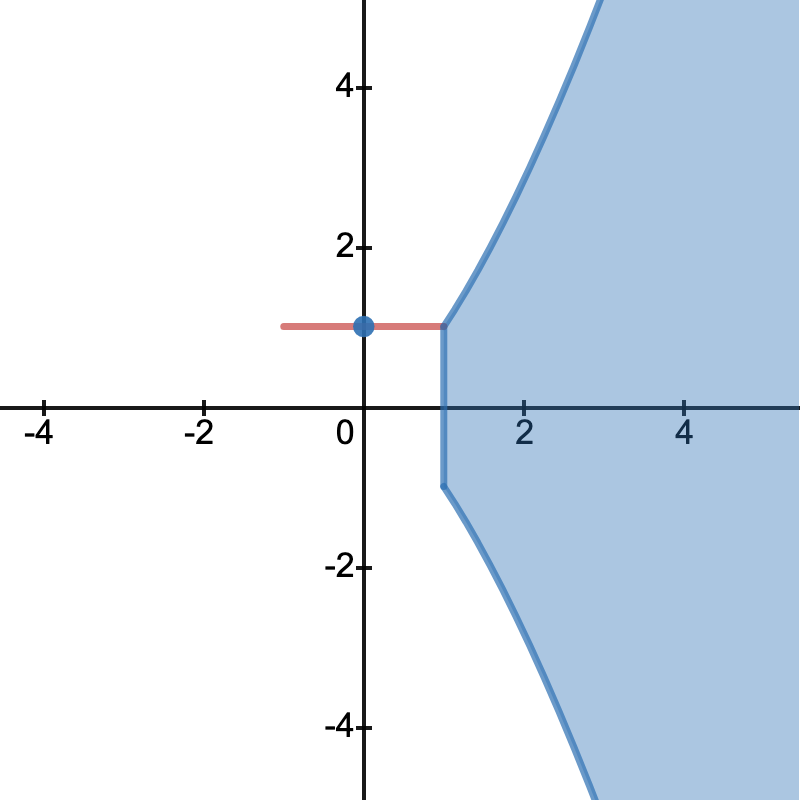}
\end{center}
The subset of $\R\T^2$ defined by \[X^3 \oplus Y \oplus -X^2 \oplus
-Y^2 \oplus -1 \geq 0.\]  The real tropicalization of
\[\{(x,y) \in K^2 : x^3+2y - x^2-y^2-1 \geq 0\}\] consists of the shaded region
and the point $(0,1)$, but not the red segment.  Compare with Figure~7 in~\cite{Ale}.
\end{minipage}
\end{tabular}
\caption{Real tropicalizations of sets and
  polynomial inequalities.}
\Label{fig:signedtrop}
\end{figure}

\begin{figure}
\begin{center}
\includegraphics[width=0.6\linewidth]{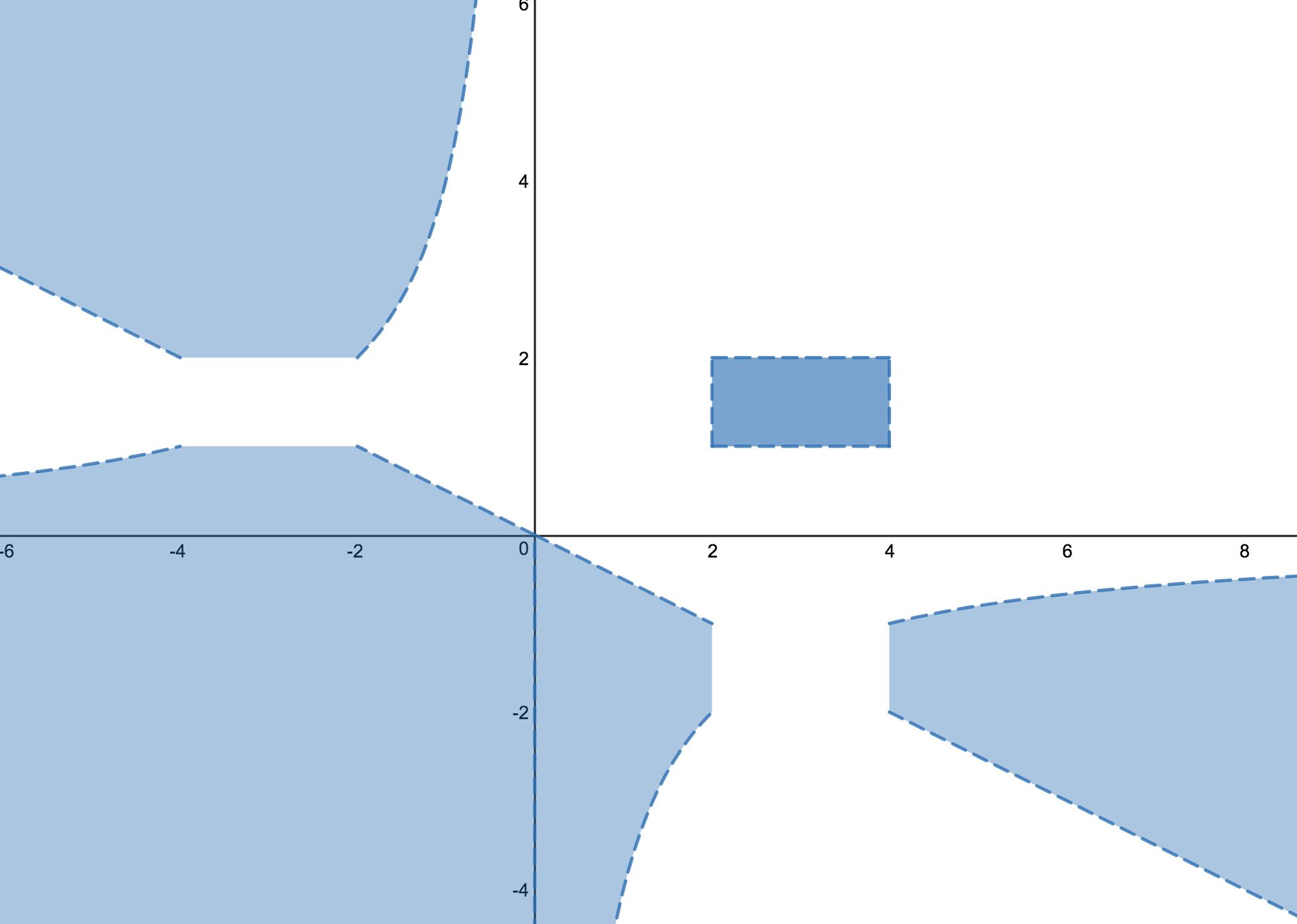}
\end{center}
\caption{The open subset of $\R^2$ defined by $$xy\left( -1 \oplus
    \frac{2}{x} \oplus \frac{x}{4} \oplus \frac{1}{y} \oplus
    \frac{y}{2} \right) < 0.$$ }
\Label{fig:rectangle}
\end{figure}


\section{The Fundamental Theorem and Limit of Tropicalizations}
\Label{sec:limit}
In this section we let $S \subset K^n$ be a semialgebraic set, with
Zariski closure $X$ in~$\A^n$.
We denote by $A$ the coordinate ring of $X$.

\subsection{Density}
\label{sec:density}

Assume that the absolute value on $K$ is non-trivial.

\begin{thm} [Weak Density] \Label{weak density}
Let $\scrF = (f_1,\dots,f_n)$ be a family of elements in~$A$.
Then $\trop_{r,\scrF}(S)$ is dense in
$\trop_{r,\scrF}(\Sanpm)$ .
\end{thm}

\begin{proof}
This follows since $S$ is dense in $\Sanpm$ (\ref{mkdense}) and
$\trop_{r,\scrF}$ is a continuous map.
\end{proof}

\begin{thm} [Strong Density] \Label{strong density}
Let $\scrF = (f_1,\dots,f_n)$ be a family of elements in $A$.
Then
\begin{align*}
\trop_{r,\scrF}(\Sanpm) \cap \left(\vert K \vert^\sgn_K\right)^m\> = \>
\trop_{r,\scrF}(S).
\end{align*}
\end{thm}

This statement is true by model-theoretic reasons (Cherlin-Dickmann
quantifier elimination for real closed fields with compatible
valuations \cite{CDII}). We nevertheless give a proof with our methods.

\begin{proof}
The inclusion $\supset$ is immediate, so we need to prove the opposite inclusion.
After possibly shrinking $S$, we may assume it is basic semialgebraic, i.e.
defined by inequalities $g_1 \geq 0,\dots,g_s \geq 0$ and $h_1 >0,\dots,h_r > 0$.

Let $z \in \trop_{r,\scrF}(\Sanpm) \cap \left(\vert K \vert^\sgn_K\right)^m$.
Let $x = \abs^\sgn_x \in \Sanpm$ such that $\trop_{\scrF}(x) = z$.
After reordering we may assume
$\vert g_1 \vert_x = 0,\dots, \vert g_t \vert_x = 0, \vert g_{t+1} \vert_x > 0, \dots, \vert g_{s} \vert_x > 0$.
Replacing $\scrF$ by $(f_1,\dots,f_m,g_1,\dots,g_t)$,
$z = (z_1,\dots,z_n)$ by
$(z_1,\dots,z_n,\vert g_1 \vert_x^\sgn,\dots,\vert g_t\vert_x^\sgn)$,
and $h_1,\dots,h_r$ by $h_1,\dots,h_r,g_{t+1},\dots,g_s$ we may assume that all inequalities are strict.

We may also assume that none of the coordinates of $z$ are zero by replacing $X$ with the vanishing locus of $f_i$
if $z_i = 0$ and $S$ by its intersection with this set.
After possibly replacing $f_i$ by $ c_i \cdot f_i$, where $\vert c_i \vert_K^\sgn = z^{-1}_i$
we may assume that $z = (1,\dots,1)$.

Consider the map $\varphi \colon \Xanpm \to \Xan$.
Then
\begin{align*}
\varphi^{-1} \{\varphi(x)\} = D \times \Sper{k_{K(x)}}
\end{align*}
by the Baer-Krull theorem \ref{Baer-Krull-Theorem}, where  $D$ is a finite discrete set.
Let $d \in D$ such that $x \in \{d\} \times \Sper{k_{K(x)}}$ and denote $E := \{d \} \times \Sper{k_{K(x)}}$.

Let
\begin{align*}
U = \{ \abs^\sgn : \vert f_i \vert^\sgn > 0 \text{ and } \vert h_j \vert^\sgn > 0 \text{ for all } i,j \} \subset \Xanpm
\end{align*}
and $V = E \cap U$.
Then $V$ is an open subset of $\Sper{k_{K(x)}}$ that contains $x$, thus there exist
$\tilde c_1,\dots,\tilde c_k \in k_{K(x)}$ such that
\begin{align*}
x \in \{ P \in \Sper{k_{K(x)}} : \tilde c_j >_P 0 \text{ and }
  \tilde f_i >_P 0 \text{ for all } i,j \} \subset V
\end{align*}
By Lemma \ref{bounded lemma} below, there exists $Q \in \Sper{k_{K(x)}}$ and $\tilde C \in k_K$ such that
\begin{align*}
\tilde C >_{Q} \tilde c_j^{\pm 1} >_Q 0 \text{ and } \tilde C
  >_{Q} \tilde f_i^{\pm 1} >_Q 0 \text{ for all } i,j.
\end{align*}
Then $>_Q$ defines a point in $V$ and consequently a point $\abs^\sgn_Q$ in $\Xanpm \subset X_r$.
Let $C \in K$ such that $\vert C \vert = 1$ and the residue class of $C$ in $k_K$ is $\tilde C$. Then
\begin{align*}
\{ P \in X_r : C^{-1} <_P f_i <_P C \text{ and } h_j >_P 0 \text{ for all } i,j\}
\end{align*}
is an open subset of $X_r$ containing $\abs_Q^\sgn$.
Thus this set is non-empty and contains a rational point $w \in X(K)$.
Since
\begin{align*}
C^{-1} < f_i(w) < C \text{ we have } \vert f_i(w) \vert^\sgn_K = 1
\end{align*}
and since $h_j(w) > 0$
meaning $w \in S$ and $\trop_{\scrF}(w) = z$.
\end{proof}

We used the following lemma:

\begin{lem} \Label{bounded lemma}
Let $F$ be a real closed field and $E$ a finitely generated extension
of $F$. Let $f_1,\dots,f_n \in E$. Assuming that there exists an order
of $E$ where all the $f_i$ are positive, there also exists an order $P$
of $E$ where all the $f_i$ are positive and bounded over $F$ (in the
sense that there exists $c \in F$ such that $c>_Pf_i$ for all $i$).
\end{lem}

\begin{proof}
Choose an affine integral $F$-variety $V$ for which $F(V)\cong E$
and $f_1,\dots,f_n\in F[V]$. The semialgebraic set $\{\xi\in V(F)\colon
f_i(\xi)>0$, $i=1,\dots,n\}$ contains a nonsingular point $\xi$ of $V$.
Since $\xi$ is nonsingular, there exists an ordering $P$ of
$F(V)$ specializing into $\xi$ inside $\Sper F[V]$. Any such $\beta$
satisfies the condition of the lemma, since $f_i<_Pc$ for
$c:=1+\max\{f_1(\xi),\dots,f_n(\xi)\}$.
\end{proof}

\subsection{Fundamental Theorem}

Let $K$ be a real closed field with a compatible absolute value (which might be trivial).
Let $L / K$ be an extension of real closed fields where the absolute value on $L$ extends the one on $K$
and is not trivial.
Let $M / L$ be an extension of real closed fields where the absolute value on $M$ extends the
one on $L$ and $\vert M \vert = \R_{\geq 0}$.

Let $S$ be a semialgebraic set defined over $K$ and let $S_M$ and $S_L$ be
the base-changes to $M$ and $L$, respectively.
Let $A$ be the coordinate ring of the Zariski closure of $S$
and $A_L := A \otimes_K L$ the base-change of $A$ to $L$.

Let $\sigma \in \{1, -1, 0\}^n$ and denote by $S^\sigma = \{ x \in S
\mid \sgn(x_i) = \sigma_i \}$ and similarly for $\R^\sigma$,
$M^\sigma$, $K^{\sigma}$ and $L^{\sigma}$.
We first show that the tropicalization of any semialgebraic set
coincides with the tropicalization of its closure.

\begin{lem} \Label{Closed vs. not closed}
For any semialgebraic set $S$ in $K^n$ and any $\sigma \in \{1,-1,0\}^n$, we have
\[
\trop_r(S^\sigma) =  \trop_r(\overline{S^\sigma}) \cap \R^\sigma.
\]
\end{lem}
\begin{proof}
The statement is immediate when the absolute value on $K$ is trivial, 
so we may assume it is non-trivial. 

The inclusion $\subset$ is obvious.

Now let $y \in \trop(\overline{S^\sigma}) \cap \R^\sigma$.
Let $x \in \overline{S^\sigma}$ such that $\trop_r(x) = y$ and let $(x^i)_{i \in \mathbb{N}}$ be
a sequence of points in $S^\sigma$ that converges to $x$.
Since the absolute value is non-trivial, the topology on $K^n$ induced by the order is
the same as the one induced by the absolute value.
Thus there exists $N \in \mathbb{N}$ such that $\vert x^i_j - x_j \vert < \vert x_j \vert$
for all $i \geq N$ and for all $j$ with $x_j \neq 0$.
By the non-archimedean triangle inequality this implies $\vert x_j^i \vert = \vert x_j \vert$ for all $i \geq N$.
In particular $y = \trop_r(x^N) \in \trop_r(S^\sigma)$.
\end{proof}

\begin{dfn}
Let $\Gamma$ be a subgroup of $\R$. 
An \emph{integral $\Gamma$-affine polyhedron} in $\R^n$ is a set of
the following form
\begin{align*}
\{ X \in \R^n \mid \alpha_i \cdot X + C_i \leq 0 \text{ for } i = 1,2,\dots,m\},
\end{align*}  
and an \emph{open} integral $\Gamma$-affine polyhedron in $\R^n$ is a set
of the form
\begin{align*}
\{ X \in \R^n \mid \alpha_i \cdot X + C_i < 0 \text{ for } i = 1,2,\dots,m\}
\end{align*}
 where $\alpha_1,\dots,\alpha_m$ are integer vectors, and $C_1,\dots,
C_m$ are in $\Gamma$.
\end{dfn}

\begin{lem} \label{lem: log polyhedron}
Let $P$ be a open $\Gamma$-affine polyhedron polyhedron in $\R^n$.
Then there exists a real tropical polynomial $F_P$ with coefficients in $\pm \exp (\Gamma)$ such that 
$P = \log \{ F_P > 0\}$. 
\end{lem}
\begin{proof}
We write
\[
P = \{ \max\left( (\alpha_1 \cdot X + C_1), \dots, (\alpha_m \cdot X + C_m) \right) > 0 \}.
\]
with $\alpha_1, \dots, \alpha_m \in \Z^n$ and $C_1,\dots,C_m \in \Gamma$.
One immediately checks that picking 
\begin{align*}
F_P = A_1 x^{\alpha_1} \oplus \cdots \oplus A_m x^{\alpha_m} \oplus (- 1)
\end{align*}
where $A_i = \exp(C_i)$ suffices. 
\end{proof}

The next two lemmas are used to prove the $(3) \subseteq (4)$ part of Theorem~\ref{thm:fund}.

\begin{lem} \label{lem polyhedra is complement}
Let $T$ be a finite union of integral $\Gamma$-affine polyhedra in $\R^n$. 
Then the complement of $T$ is a finite union of open integral $\Gamma$-affine polyhedra. 
\end{lem}
\begin{proof}
To see that $T$ has such a description, 
first write $T$ as a union of cells in an arrangement of hyperplanes defined by $f_1 = 0,
f_2 = 0, \dots, f_m = 0$ where each $f_i$ has the form $\alpha \cdot X
+ C$ with $\alpha \in \Z^n$, $C \in \Gamma$.  This can be done, for example, by
taking all the hyperplanes defining the polyhedra that make up $T$.  A
hyperplane arrangement partitions the ambient space into relatively
open polyhedra, each defined by a collection of strict inequalities and equations of
the form $f_i > 0, f_j < 0$, or $f_k = 0$. 

Let $\tau$ be a relatively open cell in the hyperplane arrangement which
is not contained in $T$.  
Let $P_\tau$ be the open polyhedron defined by all the inequalities $f_i > 0$
such that $f_i(\tau) > 0$ and $f_j < 0$ such that $f_j(\tau) < 0$.  Clearly, $\tau$ is
in $P_\tau$.  We claim that $P_\tau \cap T = \varnothing$.  
Let $\tau'$ be a cell whose relative interior is in $P_\tau$.  Then $\tau'$
is defined by the inequalities defining $P_\tau$ and possibly some additional
equations of the form $f_k = 0$.  But these additional equations must
be satisfied by $\tau$ as well, so $\tau$ is in the closure of
$\tau'$.  Since $T$ is closed and $\tau$ is
not in $T$, it follows that $\tau'$ is not in $T$.  

Thus for every relatively open cell $\tau$ which is not in $T$, we have
found an open polyhedron $P_\tau$ which contains $\tau$ and does not intersect $T$.  Since the
hyperplane arrangement consists of finitely many 
cells, we have found a description of $T$ as the intersection of
complements of finitely many open polyhedra.
\end{proof}

Next, we will prove a lifting lemma for inequalities.
\begin{lem}\label{lem:lift}
Let $S \subset K_{>0}^n$ be a semialgebraic set and $F$ be a tropical
polynomial with coefficients in $|K|^\sgn_K$ such that $F \geq 0$ on
$\trop_r(S)$.  Then there is a polynomial $f$ in $K[x_1,\dots,x_n]$ such that $f \geq 0$ on $S$ and
$\trop(f) = F$.
\end{lem}
\begin{proof}
Suppose $F \geq 0$ on $\trop_r(S)$, and $F$ has the form
 \[F(X_1,\dots,X_n) = A_1 X^{\alpha_1} \oplus \cdots \oplus A_k X^{\alpha_k}
\oplus B_1 X^{\beta_1} \oplus \cdots \oplus B_\ell X^{\beta_\ell}\] 
where $A_1,\dots,A_k$ are positive elements and $B_1,\dots,B_\ell$ are
negative elements of $|K|^\sgn_K$.
  Consider a polynomial
of the form \[f_\varepsilon(x_1,\dots,x_n) = a_1 x^{\alpha_1} + \cdots + a_k x^{\alpha_k}
+ \varepsilon( b_1 x^{\beta_1} + \cdots + b_\ell x^{\beta_\ell})\]
where $|a_i|^\sgn_K = A_i$ and $|b_j|^\sgn_K = B_j$ for each
$i=1,\dots, k$ and $j = 1,\dots,\ell$.  We wish to show
that $f_\varepsilon > 0$ on $S$ for some positive $\varepsilon \in K$
with $|\varepsilon|_K = 1$.

Consider the map 
\[ \varphi : S \rightarrow K, \quad~ x \mapsto \frac{ a_1
x^{\alpha_1} + \cdots + a_k x^{\alpha_k}}{-( b_1 x^{\beta_1} +
\cdots + b_\ell x^{\beta_\ell})}. \]
For any $x \in S \subseteq K_{>0}^n$, 
we have $|x|^\sgn_K \in \Trop_r(S)$, so $F(|x|^\sgn_M) \geq 0$. 
Then we have 
\[ |a_1 x^{\alpha_1} + \cdots + a_k x^{\alpha_k}|^\sgn_K ~~\geq~~ 
|-( b_1 x^{\beta_1} + \cdots + b_\ell x^{\beta_\ell})|^\sgn_K ~~>~~ 0,\]
which implies that $|\varphi(x)|^\sgn_K \geq 1$.  
The image
$\varphi\left( S \right)$ is a $K$-semialgebraic subset of $K$, so
there is an element $0 < \varepsilon \in K$ with $|\varepsilon|_K = 1$ 
such that $\varphi\left( S\right) \geq \varepsilon$.  
This means $f_\varepsilon > 0$ on $S$.
\end{proof}

We will now give an analogue of the Fundamental Theorem of Tropical
Geometry, for semialgebraic sets.  Compare with~\cite[Theorem 4.2]{Draisma}, \cite{EKL},
and \cite[\S3.2]{TropicalBook}.

\begin{thm} \Label{thm:fund}  Let $S$ be a semialgebraic set in $K^n$ defined over $K$.
The following subsets of $\R^n$ agree.
\begin{enumerate}
\item
The tropicalization of the real analytification $\Sanpm$ of $S$:
\[\trop_r(\Sanpm);\]
\item
The tropicalization of the $M$-points of $S$:
\[\trop_r(S_M);\]
\item The set defined by tropicalization of weak
  inequalities in $K[T_1,\dots,T_n]$ satisfied by $S$ in each orthant:
\[
\bigcup_{\sigma \in \{+,-,0\}^n} \left( \bigcap_{\substack{f \geq
    0 \text{ on }S^\sigma}} \{\trop(f) \geq 0\}
  \cap \R^\sigma \right);
\]
 The intersection may be taken to be finite.
\item
The orthantwise closure of the tropicalization of the $L$-points of $S$:
\begin{align*}
\bigcup_{\sigma \in \{+,-,0\}^n}  {\overline{\trop_r(S_L^\sigma})} \cap \R^{\sigma}.
\end{align*}
\end{enumerate}  
If $S$ is closed, then all these agree with the closure of
$\trop_r(S_L)$ in $\R^n$. 
\end{thm}

We will denote by $\Trop_r(S)$ the set described in the theorem above.  
The descriptions (2) and (4) are the most commonly found in the
literature~\cite{Ale, AGS}.  

In (4), if $S$ is not closed, we need to take the closure in each orthant separately. For
example, for $S = K \setminus \{0\}$, we have $\Trop_r(S) = \R
\setminus \{0\}$, which is not closed in $\R$ but its intersection
with $\R_{>0}$, $\R_{<0}$, and $\{0\}$ are all closed.  
 In (3) it is necessary to deal with each orthant separately, as we will see in Remark~\ref{rmk:orthants}.

For $S$ in the positive orthant defined over $\R$, Alessandrini showed that the logarithm of 
$\Trop_r(S)$ agrees with the logarithmic limit set of $S$~\cite[Theorem~4.2]{Ale}, and that $S$ has a description whose
tropicalization characterizes $\Trop_r(S)$ \cite[Corollary~6.7]{Ale}, where inequalities can be joined by both {\em and} and {\em or}.  
But in description (3) of the Theorem above, we use only {\em and} in each orthant.  

It follows from the equivalence of (2) and (4) that the image of any
semialgebraic set under the tropicalization map is closed in each
orthant, when the value group of the field $M$ is $\R$.  See also~\cite[Theorem~10]{AGS}.

\begin{proof}
Since the map $\trop_r$ respects orthants in the obvious way and (3) and (4) are constructed orthantwise,
to prove the equivalence of (1) - (4), we may assume $S = S^\sigma$ for some fixed $\sigma \in \{\pm 1 , 0\}^n$.

\underline{$(1) \subset (2):$} Since there is a surjection $(S_M)_r^{\an} \to \Sanpm$ we have
$\trop_r((S_M)_r^{\an}) \supset \trop_r(S_r^{\an})$.
By Theorem \ref{strong density} we have $\trop_r((S_M)_r^{\an}) = \trop_r(S_M)$.

\underline{$(2) \subset (3):$} This follows from Section~\ref{sec:ineq}.

\underline{$(3) \subset (4)$:}   Without loss of generality, suppose
$S \subset K_{>0}^n$.  By Theorem 4 of~\cite{AGS},
$T := \log\left(\trop_r(S_L)\right)$ is a finite union of integral $\log \vert K \vert$-affine polyhedra.
By Lemma \ref{lem polyhedra is complement} 
we have $T = \R^n \setminus \cup P_i = \cap (\R^n \setminus P_i)$ 
for open integral $\log \vert K \vert$-affine polyhedra $P_i$. 
Taking real tropical polynomials $F_{P_i}$ as in Lemma~\ref{lem polyhedra is complement}, 
we find $\overline{\left(\trop_r(S_L)\right)} = \exp(T) = \cap \{ - F_{P_i} \geq 0 \}$. 
Now the statement follows by taking polynomials $f_{P_i}$ over $K$
such that $\trop (f_{P_i}) = F_{P_i}$ and $f_{P_i} \geq 0$ on $S$ as
in Lemma~\ref{lem:lift}.  Finiteness follows from the finiteness in Lemma \ref{lem polyhedra is complement}.



\underline{$(4) \subset (1)$:}
We first prove the statement assuming $K = L$; then $S = S_L$. 

We have the following inclusions, where all closures are taken within the orthant corresponding to $\sigma$:
\begin{align*}
\overline{\trop_r S} \subset \overline{\trop_r(\Sanpm)} = \trop_r (\overline{\Sanpm}) \subset \trop_r(\overline{S})
\subset \trop_r(S) \subset \trop_r(\Sanpm).
\end{align*}
The first and last inclusion follow from the inclusion $S \subset \Sanpm$. 
The equality follows because the map $\trop_r \colon (\A_r^{n, \an})^\sigma \to \R^{\sigma}$
is proper by Lemma \ref{trop proper}, so in particular closed.  
The next inclusion $\trop_r (\overline{\Sanpm}) \subset \trop_r(\overline{S})$ is $(1) \subset (4)$, 
applied to $\overline{S}$. 
Note here that $(\overline{S})^{\an}_r = \overline{\Sanpm}$. 
Lastly the inclusion $\trop_r(\overline{S}) \subset \trop_r(S)$ 
is Lemma \ref{Closed vs. not closed}.

Now we have to prove the general case, knowing already $\overline{\trop_r(S_L)} \subset \trop_r( (S_L)^{\an}_r)$. 
Noting that obviously $(S_L)_M = S_M$ we find
\begin{align*}
\trop_r( (S_L)^{\an}_r) \subset \trop_r(S_M) \subset \trop_r(S_r^{\an}).
\end{align*}
Here the first inclusion is $(1) \subset (2)$ applied to $S_L$ and the 
second inclusion comes from the fact that the canonical map $S_M \to S_r^{\an}$ 
is compatible with the tropicalization. 
This proves $(4) \subset (1)$ and completes the proof of the equality of $(1) - (4)$. 

For the last sentence in the theorem, observe that if $S$ is closed, then $\Sanpm$ is closed.
Then $\trop_r(\Sanpm)$ is closed by properness of the map $\trop_r \colon \A_r^{n, \an} \to \R^n$
and the claim follows from the Weak Density Theorem \ref{weak density} and
the inclusion $\trop_r(S_L) \subset \trop_r(\Sanpm)$.
\end{proof}

We have the following Corollary,
using Theorem~\ref{thm:connected}. 

\begin{cor}
\Label{cor:tropConnected}
If $S$ is semialgebraically connected, then $\Trop_r(S)$ is
connected.
\end{cor}

\begin{rem}
\label{rmk:orthants}
In the description (3) above, it is necessary to consider each orthant
separately if $n \geq 2$.  That is, it cannot be replaced by
\[
\bigcap_{f \geq
    0 \text{ on } S} \{\trop(f) \geq 0\}.
\]
To see this, let $S$ be the complement of the positive orthant in $K^n$.  We claim
that for a polynomial $f$, if $S \subseteq \{f \geq 0\}$,
then $\{\trop(f) \geq 0\} = \R\T^n$; therefore the intersection above is
all of $\R\T^n$.

Let $F$ be a tropical polynomial such that $\{F < 0\} \subset \R\T_{>0}^n$. 
We will show that $\{F < 0 \} = \varnothing$, so $\{F \geq 0\} = \R\T^n$.
Suppose not.  
Then there exists $y \in \{F < 0\}$ where the maximum absolute value among terms in $F(y)$ is attained uniquely at a term
\[c y_1^{a_1} y_2^{a_2} \cdots y_n^{a_n} < 0.\]
Since $\{F < 0\}$ is contained in the positive orthant, we have
\begin{align*}
c (-y_1)^{a_1} y_2^{a_2}\cdots y_n^{a_n}
&> 0 \\
c y_1^{a_1} (-y_2)^{a_2}\cdots y_n^{a_n} & > 0 \\
c (-y_1)^{a_1} (-y_2)^{a_2}\cdots y_n^{a_n} & > 0,
\end{align*}
which is a contradiction since the product of the left hand sides of the four
inequalities above must be positive.
\qed
\end{rem}

\begin{prop}
\label{prop:intersection}
Let $S$ be a semialgebraic set of the form $S = \bigcap \{ f_i \leq
0\}$. 
Then we have the following chain of inclusions
\begin{align*}
\interior{\Trop_r(S)} \subset \interior {\bigcap \{ \trop_r(f_i) \leq 0\}} \subset
\Trop_r(S) \subset \bigcap \{ \trop_r(f_i) \leq 0\},
\end{align*}
where $\Int$ denotes the interior.  
\end{prop}

This Proposition has two consequences for computing the 
tropicalization of $ S = \bigcap \{ f_i \leq0\}$.
Firstly, while one may get too big of a set by simply tropicalizing all inequalities
defining $S$, the difference will have empty interior. 
Secondly, if $\bigcap \{ \trop_r(f_i) \leq 0\}$ is the closure of its interior, 
then it agrees with $\Trop_r(S)$ since $\Trop_r(S)$ is closed.
Also see Corollary 16 of~\cite{AGS}.

\begin{proof}
The last inclusion follows from the Fundamental Theorem \ref{thm:fund}. 
The first inclusion follows from the last one by taking interiors. 
For the middle inclusion, we first note that 
\begin{align*}
\left(\bigcap \{ \trop_r(f_i) \leq 0 \} \right) \setminus \left(\bigcap( \{ \trop_r(f_i) < 0 \}) \right) \subset 
\bigcup \{ 0 \in \trop_r(f_i) \}.
\end{align*}
The last set is contained in an algebraic subset of $\R^n$, hence has empty interior. 
We conclude that 
\begin{align*}
\interior {\bigcap\{ \trop_r(f_i) \leq 0 \} } \setminus \overline{\bigcap \{ \trop_r(f_i) < 0 \} } = \emptyset
\end{align*}
that is, $\interior { \bigcap\{ \trop_r(f_i) \leq 0 \} } \subset \overline{ \bigcap\{ \trop_r(f_i) < 0 \} }$. 
The last set is however contained in $\Trop_r(S)$ by the facts that
$\trop_r(f)(|z|^\sgn) > 0$ implies $f(z) > 0$ and that $\Trop_r(S)$ is
closed since $S$ is closed. 
\end{proof}

\begin{example}
\label{ex:intersection}
Let  $f(x,y) = 1 - (x-y)^2$, $g_c(x,y) = 2x^2+2x-xy+cy^2$.   
For all values of $c$ with $|c|^\sgn =
-1$, the tropicalization
 $\Trop_r(\{g_c \geq 0\})$ is the same, but the tropicalization of the intersection $\{f \geq 0\} \cap \{g_c \geq 0\}$
depends on the choice of $c$.  The cases of $c < -1$ with $|c|^\sgn =
-1$ are depicted in
Figure~\ref{fig:trop_intersection}.  For $-1 \leq c < 0$  with $|c|^\sgn =
-1$, the
tropicalization of the intersection contains the two infinite rays as
$a \rightarrow \infty$ in the figure.
  (The dotted curves together with
  the boundary of the regions form the set $\{(X,Y) : X^2 + X - XY -
  Y^2 \ni 0\}$.)
 In any case, difference  $\Trop_r(\{f \geq 0\}) \cap \Trop_r(\{g_c \geq
0\}) \setminus \Trop_r(\{f \geq 0\} \cap \{g_c \geq 0\})$ is a set
with empty interior. 
\qed
\end{example}

\begin{figure}
\begin{tabular}{ccc}
\begin{minipage}{0.3\textwidth}
\includegraphics[width=\linewidth]{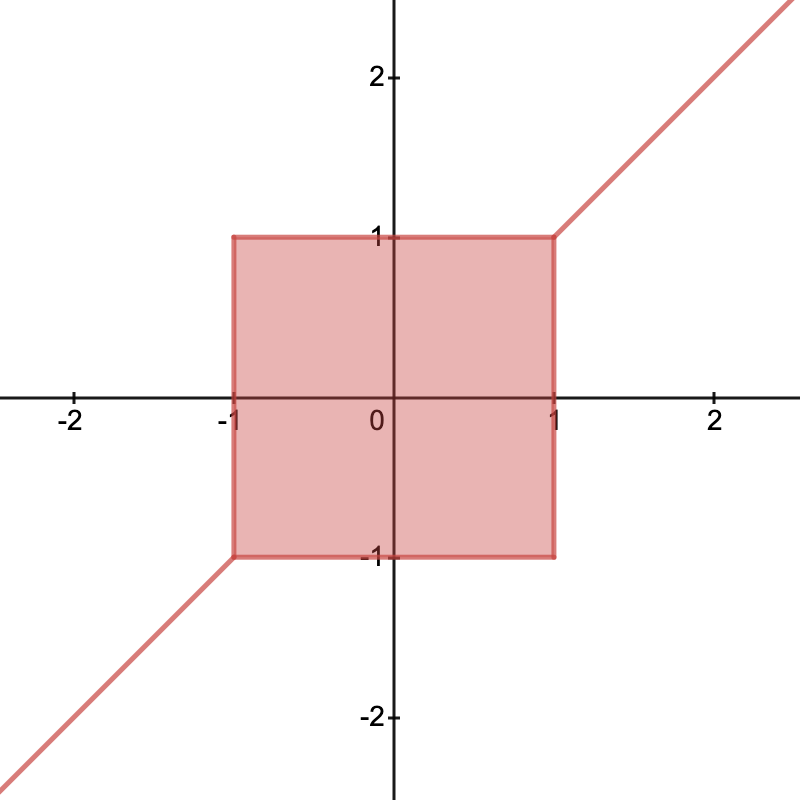}
\end{minipage} &
\begin{minipage}{0.3\textwidth}
\includegraphics[width=\linewidth]{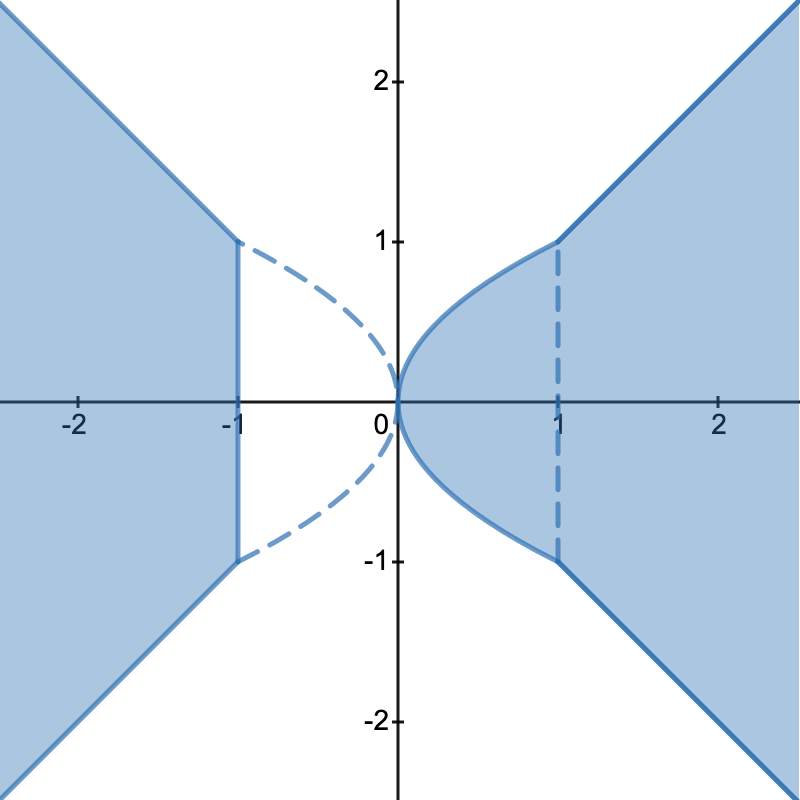}
\end{minipage} &
\begin{minipage}{0.3\textwidth}
\includegraphics[width=\linewidth]{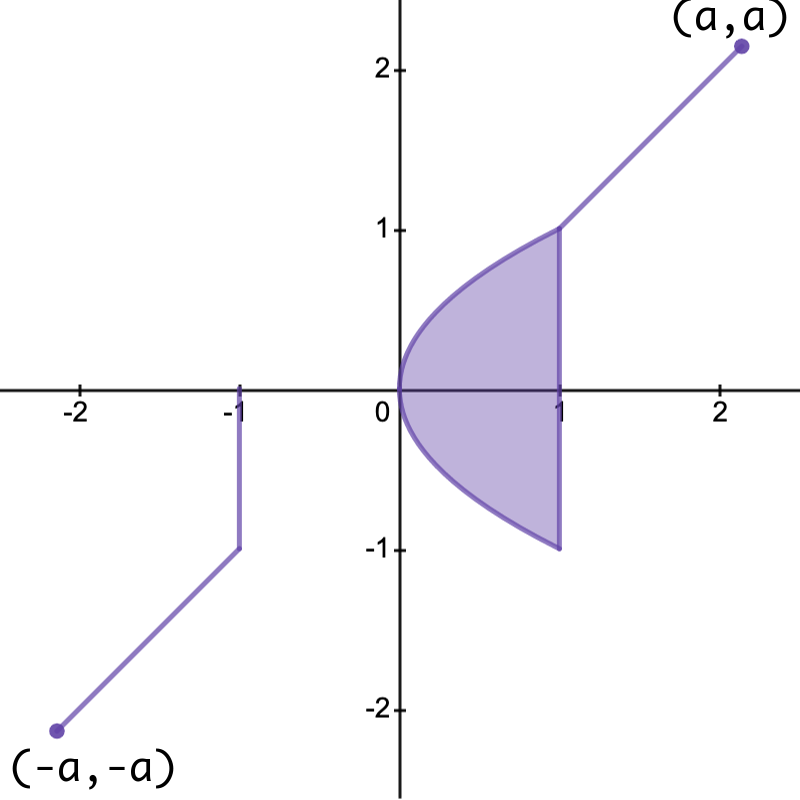}
\end{minipage}
\end{tabular}
\caption{The left and middle figures show the tropicalizations of the
  sets $\{1 - (x-y)^2 \geq 0\}$ and $\{2x^2+2x-xy+cy^2 \geq 0\}$
  respectively, where $|c|^\sgn = -1$.  
The figure on the right shows
  the tropicalization of the intersection where $|c|^\sgn = -1$, $c < -1$ and \mbox{$a =
    -\frac{1}{|c+1|^\sgn}$}.  See Example~\ref{ex:intersection}.} 
\label{fig:trop_intersection}
\end{figure}

\subsection{Limit Theorem}

 The following is
the semialgebraic and real analogue of Payne's theorem in~\cite{p}.

For any finite sequence $\scrF=(f_1,\dots,f_m)$ in $K[x_1,\dots,x_n]$,
let $\Trop_{r,\scrF}(S)\subset\R^m$ be the image of the map
\begin{align*}
\trop_{r,\scrF}\colon \Sanpm \to\R^m,\quad x\mapsto
\trop_{r,\scrF}(x):=(|f_1|_x^\sgn,\dots,|f_m|_x^\sgn).
\end{align*}

Given a second such sequence $\scrF'=(f'_1,\dots,f'_r)$ and a map
$\tau \colon \{1,\dots,m\}\to\{1,\dots,r\}$ satisfying $f'_{\tau(i)} = f_{i}$, the corresponding projection
$\R^r\to\R^m, e_{\tau(i)} \mapsto e_i$ induces a map $\Trop_{r,\scrF'}(S)\to\Trop_{r,\scrF}(S)$.
Passing to the limit over all such finite sequences (with respect to
maps between the indexing sets), we get the following analogue
of Payne's theorem \cite[Theorem 1.1]{p}:

\begin{thm} \Label{signed limit}
The canonical map
$$\Sanpm\ \to\ \varprojlim_\scrF\Trop_{r,\scrF}(S),\quad x\ \mapsto\
\trop_{r,\scrF}(x)$$
is a homeomorphism, where the limit is taken over the finite
families $\scrF$ in $K[x_1,\dots,x_n]$ as above.
\end{thm}

\begin{proof}
The map is injective, since if $x\ne y$, there exists a polynomial
$f$ with $\vert f \vert^\sgn_x \neq \vert f \vert^\sgn_y$. Consequently,
the images of $x$ and $y$ disagree in $\Trop_{r,\scrF}(S)$ for any $\scrF$ containing $f$.

To show that the map is surjective, let $z \in \varprojlim_\scrF\Trop_{r,\scrF}(S)$; 
it assigns an element $z_\scrF \in \Trop_{r,\scrF}(S)$ to every finite family $\scrF$ of polynomials, satisfying compatibility conditions.  
Define a signed seminorm $\abs_x^\sgn$
as follows. For $f\in K[x_1,\dots,x_n]$, let
$$|f|_x^\sgn\>:=\>(z_\scrF)_1,$$
the first coordinate of $z_\scrF$, where $\scrF=(f,f_2,\dots)$ is any family whose first component is $f$. 
This is well-defined since $z_\scrF$ are compatible with maps between the indexing sets as described above.

To see that $\abs_x^\sgn$ is a seminorm,
we need to check the following, for all polynomials $f$ and $g$:
\begin{enumerate}
\item
$\vert f \vert^\sgn_x = \sgn(f) \vert f \vert_K$ if $f \in K$,
\item
$\vert f \cdot g \vert^\sgn_x = \vert f \vert^\sgn_x \cdot \vert g \vert^\sgn_x$,
\item
$\min(\vert f \vert^\sgn_x, \vert g \vert^\sgn_x)\ \le\ \vert f + g \vert^\sgn_x
\ \le\ \max(\vert f \vert^\sgn_x, \vert g \vert^\sgn_x)$.
\end{enumerate}
Take the tuple $((z_\scrF)_1,\dots,(z_\scrF)_4)$ corresponding to the
family $\scrF=(f,g, f+g, fg)$, and let $y\in\Sanpm$ be such that
$\trop_{r,\scrF}(y)=((z_\scrF)_1,\dots,(z_\scrF)_4)$.
Then $\vert f \vert^\sgn_x = \vert f \vert^\sgn_y$,
$\vert g \vert^\sgn_x = \vert g \vert^\sgn_y$,
$\vert f + g \vert^\sgn_x = \vert f +g \vert^\sgn_y$ and
$\vert fg \vert^\sgn_x = \vert f \cdot g \vert^\sgn_y$,
and the axioms $(1)$, $(2)$ and $(3)$ hold for $\abs_x^\sgn$ because
they hold for $\abs_y^\sgn$.

It remains to show that $\abs_x^\sgn \in \Sanpm$. Let us write $S = S_1 \cup \cdots S_n$ as a union of basic semialgebraic subsets. 
This means that each $S_i$ is defined by finitely many inequalities 
$f_{ij} \geq 0$ and $g_{ik} > 0$ where $j=1,\dots, M_i, k=1,\dots,
N_i$.  Let $\scrF$ be the family consisting all $f_{ij}$ and $g_{ik}$ and consider $z_\scrF$.  Let 
$y = \abs_y^\sgn \in S_r^{\an}$ such that $\trop_{r, \scrF}(y) = z_\scrF$.
Then $\abs_y^\sgn \in (S_\ell)^{\an}_r$ for some $\ell$, and 
it satisfies all the inequalities $\vert f_{\ell j} \vert_y^\sgn \geq 0$ and $\vert g_{\ell j} \vert_y^\sgn > 0$. 
Since we have $\vert f_{\ell j} \vert_x^{\sgn} = (z_{\scrF})_{\ell j} = \vert f_{\ell j} \vert_y^\sgn$ 
and similarly for the $g_{\ell j}$, we see that $\abs_x^\sgn \in (S_\ell)_r^{\an} \subseteq \Sanpm$. 

Finally the map is a homeomorphism because the topology on the
left is defined as the coarsest topology such that
$x \mapsto \vert f \vert^\sgn_x$ is continuous for all
$f$, while on the right the topology is defined such that
all projection maps, i.e.~all maps $\varprojlim_\scrF\Trop_{r,\scrF}(S)
\to\Trop_{r,\scrF'}(S)$ to a particular $\scrF'$ are continuous.
These conditions are  equivalent.
\end{proof}

\section*{Acknowledgements}
This work started when CS and JY were in residence at
the Mathematical Sciences Research Institute in Berkeley, California,
during the Fall 2017 semester on Geometric and Topological Combinatorics, supported by the National Science
Foundation under Grant No.\ DMS-1440140.
PJ and JY are grateful to the Institut Mittag-Leffler in Djursholm,
Sweden, for hospitality during the Spring 2018 semester.
JY thanks the Institute for Computational and Experimental Research in
Mathematics in Providence, Rhode Island, for support during the Fall
2018 semester.  
We thank the referees for helpful comments.

\bibliographystyle{amsalpha}
\providecommand{\bysame}{\leavevmode\hbox to3em{\hrulefill}\thinspace}
\providecommand{\MR}{\relax\ifhmode\unskip\space\fi MR }
\providecommand{\MRhref}[2]{%
  \href{http://www.ams.org/mathscinet-getitem?mr=#1}{#2}
}
\providecommand{\href}[2]{#2}

\end{document}